\documentclass[pagesize,pdftex]{scrartcl}
\usepackage{latexsym} 
\usepackage[intlimits]{amsmath}
\usepackage{amsthm}
\usepackage{amsfonts}
\usepackage{amssymb}
\usepackage{amsxtra}
\usepackage{mathtools}
\usepackage{amscd}
\usepackage{ifthen}
\usepackage{graphicx}
\usepackage{color}
\usepackage[shortlabels]{enumitem}
\usepackage{mathrsfs}
\usepackage[arrow, matrix, curve]{xy}
\usepackage[pagebackref=true]{hyperref}
\usepackage{url}
\usepackage[utf8]{inputenc}
\usepackage{doi}

\hypersetup{
pdfauthor={Aday Celik, Mads Kyed},
pdftitle={Time-periodic Stokes equations with inhomogeneous Dirichlet boundary conditions in a half-space},
breaklinks=true,
colorlinks=true,
linkcolor=blue,
citecolor=blue,
urlcolor=blue,
filecolor=blue,
}

\numberwithin{equation}{section} 
\setkomafont{title}{\normalfont}

\newenvironment{pdeq}{ \left\{ \begin{aligned}}{\end{aligned}\right.}

\newcommand{\eqrefsub}[2]{\eqref{#1}\textsubscript{#2}}
\newcommand{\np}[1]{(#1)}
\newcommand{\nb}[1]{[#1]}
\newcommand{\bp}[1]{\big(#1\big)}
\newcommand{\bb}[1]{\big[#1\big]}
\newcommand{\Bp}[1]{\bigg(#1\bigg)}

\newcommand{\Bb}[1]{\bigg[#1\bigg]}
\newcommand{\calp}{{\mathcal P}}

\newcommand{\calt}{{\mathcal T}}

\newcommand{\R}{\mathbb{R}}
\newcommand{\C}{\mathbb{C}}
\newcommand{\Z}{\mathbb{Z}}

\newcommand{\N}{\mathbb{N}}
\DeclareMathOperator{\e}{e}

\newcommand{\linearmaps}{\mathscr{L}}

\DeclareMathOperator{\Div}{div}

\DeclareMathOperator{\supp}{supp}

\DeclareMathOperator{\trace}{Tr}

\newcommand{\mult}{\mathfrak{m}}
\newcommand{\tmult}{\mathfrak{\widetilde{m}}}

\newcommand{\ceq}{\coloneqq}
\newcommand{\ra}{\rightarrow}

\newcommand{\set}[1]{\ensuremath{\{#1\}}}

\newcommand{\setc}[2]{\ensuremath{\{#1\ \lvert\ #2\}}}
\newcommand{\setcl}[2]{\ensuremath{\bigl\{#1\ \lvert\ #2\bigr\}}}
\newcommand{\setcL}[2]{\ensuremath{\biggl\{#1\ \lvert\ #2\biggr\}}}

\newcommand{\closure}[2]{\overline{#1}^{#2}}
\newcommand{\proj}{\calp}
\newcommand{\projcompl}{\calp_\bot}

\newcommand{\quotientmap}{\pi}
\newcommand{\grp}{{\torus\times\Rn}}
\newcommand{\grpx}{{\torus\times\Rn}}
\newcommand{\grpxx}{{\torus\times\ws}}
\newcommand{\Zgrp}{{\frac{2\pi}{\per}\Z}}
\newcommand{\dualgrp}{{\frac{2\pi}{\per}\Z\times\Rn}}

\newcommand{\grpH}{H}
\newcommand{\grpG}{G}
\newcommand{\dualgrpH}{\widehat{H}}
\newcommand{\dualgrpG}{\widehat{G}}

\newcommand{\torus}{{\mathbb T}}

\newcommand{\halfspace}{\R^n_+}
\newcommand{\wholespace}{\R^{n-1}}
\newcommand{\hs}{{\halfspace}}
\newcommand{\ws}{\wholespace}
\newcommand{\Rn}{{\R^n}}

\newcommand{\grad}{\nabla}

\newcommand{\px}{\partial_x}
\newcommand{\pt}{\partial_t}

\newcommand{\dx}{{\mathrm d}x}

\newcommand{\dg}{{\mathrm d}g}

\newcommand{\dt}{{\mathrm d}t}

\newcommand{\dxi}{{\mathrm d}\xi}

\newcommand{\SR}{\mathscr{S}}
\newcommand{\SRh}{Z}

\newcommand{\TDR}{\mathscr{S^\prime}}
\newcommand{\TDRh}{\SRh^\prime}

\newcommand{\ft}[1]{\widehat{#1}}

\newcommand{\FT}{\mathscr{F}}
\newcommand{\iFT}{\mathscr{F}^{-1}}

\newcommand{\mmultiplier}{m}
\newcommand{\Mmultiplier}{M}
\newcommand{\MmultiplierRestriction}{M_{|\Zgrp\times\ws}}
\newcommand{\MmultiplierNr}[1]{M_{#1}}

\newcommand{\multrestriction}[2]{#1_{|#2}}

\newcommand{\norm}[1]{\lVert#1\rVert}

\newcommand{\normL}[1]{\Bigl\lVert#1\Bigr\rVert}
\newcommand{\snorm}[1]{{\lvert #1 \rvert}}

\newcommand{\opnorm}[1]{{\vert\kern-0.25ex\vert\kern-0.25ex\vert #1 \vert\kern-0.25ex\vert\kern-0.25ex\vert}}
\newcommand{\parnorm}[2]{|#1,#2|}
\newcommand{\WSR}[2]{W^{#1,#2}}
\newcommand{\WSRD}[2]{\dot{W}^{#1,#2}} 
\newcommand{\WSRN}[2]{W^{#1,#2}_0} 
\newcommand{\HSR}[2]{H^{#1,#2}}
\newcommand{\HSRD}[2]{\dot{H}^{#1,#2}}

\newcommand{\CR}[1]{C^{#1}}  
\newcommand{\LR}[1]{L^{#1}}

\newcommand{\CRi}{\CR \infty}
\newcommand{\CRci}{\CR \infty_0}

\newcommand{\BSRcompl}[2]{B^{#1}_{#2,\bot}}
 
\newcommand{\ABPSRcompl}[2]{H^{#1}_{#2,\bot}}

\newcommand{\LRper}[1]{L^{#1}_{\mathrm{per}}}
\newcommand{\WSRper}[2]{W^{#1,#2}_{\mathrm{per}}} 

\newcommand{\CRiper}{\CR{\infty}_{\mathrm{per}}}

\newcommand{\vvel}{v}
\newcommand{\vpres}{\Pi}
\newcommand{\uvel}{u}
\newcommand{\upres}{p}
\newcommand{\tuvel}{\tilde{u}}
\newcommand{\tupres}{\tilde{p}}

\newcommand{\adjvel}{\psi}
\newcommand{\adjpres}{\eta}
\newcommand{\wvel}{w}
\newcommand{\wpres}{\pi}
\newcommand{\twvel}{\widetilde{w}}

\newcommand{\Uvel}{U}

\newcommand{\Upres}{\mathfrak{P}}

\newcommand{\uvelp}{v}
\newcommand{\uveln}{w}
\newcommand{\tin}{\text{in }}
\newcommand{\tif}{\text{if }}
\newcommand{\ton}{\text{on }}

\newcommand{\half}{\frac{1}{2}}

\renewcommand{\epsilon}{\varepsilon}

\newcommand{\tay}{\calt}
\newcommand{\per}{\tay}
\newcommand{\iper}{\frac{1}{\tay}}

\newcommand{\dualrho}{\hat{\rho}}

\newcommand{\pex}{\parnorm{\eta}{\xi}}
\newcommand{\f}{f}
\newcommand{\h}{h}
\newcommand{\rhsh}{H}
\newcommand{\rhsG}{G}
\newcommand{\rhshp}{\rhsh^\prime}
\newcommand{\rhshn}{\rhsh_n}
\newcommand{\g}{g}

\newcommand{\qdf}{q_0}
\newcommand{\qdfone}{q_1}
\newcommand{\qdftwo}{q_2}

\newcommand{\qoprgood}{\mathscr G}
\newcommand{\qoprbad}{\mathscr B}
\newcommand{\pu}{\phi}
\newcommand{\hf}{h}
\newcommand{\newCCtr}[2][d]{
\newcounter{#2}\setcounter{#2}{0}
\expandafter\xdef\csname kyedtheconst#2\endcsname{#1}
}
\newcommand{\Cc}[2][nolabel]{
\stepcounter{#2}
\expandafter\ensuremath{\csname kyedtheconst#2\endcsname_{\arabic{#2}}}
\ifthenelse{\equal{#1}{nolabel}}
{}
{\expandafter\xdef\csname kyedconst#1\endcsname
{\expandafter\ensuremath{\csname kyedtheconst#2\endcsname_{\arabic{#2}}}}}
}
\newcommand{\Ccn}[2][nolabel]{
\expandafter\ensuremath{\csname kyedtheconst#2\endcsname}
\ifthenelse{\equal{#1}{nolabel}}
{}
{\expandafter\xdef\csname kyedconst#1\endcsname
{\expandafter\ensuremath{\csname kyedtheconst#2\endcsname}}}
}
\newcommand{\CcSetCtr}[2]{
\setcounter{#1}{#2}
}
\newcommand{\Cclast}[1]{
\expandafter\ensuremath{\csname kyedtheconst#1\endcsname_{\arabic{#1}}}
}
\newcommand{\Ccllast}[1]{
\addtocounter{#1}{-1}
\expandafter\ensuremath{\csname kyedtheconst#1\endcsname_{\arabic{#1}}}
\addtocounter{#1}{1}
}
\newcommand{\const}[1]{
\expandafter{\ifcsname kyedconst#1\endcsname
  \csname kyedconst#1\endcsname
\else
  \errmessage{Undefined Kyedconstant #1.}%
\fi}
}
\newcommand{\Ccnlast}[1]{
\Ccn{#1}
}

\theoremstyle{plain}
\newtheorem{thm}{Theorem}[section]
\newtheorem{defn}[thm]{Definition}
\newtheorem{lem}[thm]{Lemma}
\newtheorem{prop}[thm]{Proposition}

\theoremstyle{remark}
\newtheorem{rem}[thm]{Remark}

\begin{document}

\title{Time-periodic Stokes equations with inhomogeneous Dirichlet boundary conditions in a half-space}

\author{
Aday Celik\\ 
Fachbereich Mathematik\\
Technische Universit\"at Darmstadt\\
Schlossgartenstr. 7, 64289 Darmstadt, Germany\\
Email: {\texttt{celik@mathematik.tu-darmstadt.de}}
\and
Mads Kyed\\ 
Fachbereich Mathematik\\
Technische Universit\"at Darmstadt\\
Schlossgartenstr. 7, 64289 Darmstadt, Germany\\
Email: {\texttt{kyed@mathematik.tu-darmstadt.de}}
}

\date{\today}
\maketitle

\begin{abstract}
The time-periodic Stokes problem in a half-space with fully inhomogeneous right-hand side is investigated. 
Maximal regularity in a time-periodic Lp setting is established. A method based on Fourier-multipliers
is employed that leads to a decomposition of the solution into a steady-state and a purely oscillatory part in
order to identify the optimal function spaces.
\end{abstract}


\newCCtr[C]{C}
\newCCtr[M]{M}
\newCCtr[\epsilon]{eps}
\CcSetCtr{eps}{-1}
\newCCtr[c]{c}
\let\oldproof\proof
\def\proof{\CcSetCtr{c}{-1}\oldproof}

\section{Introduction}
We investigate the $\per$-time-periodic Stokes problem
\begin{align}\label{SH}
\begin{pdeq}
\pt\uvel - \Delta\uvel + \grad \upres &= \f && \tin\R\times\halfspace, \\
\Div\uvel &= \g && \tin\R\times\halfspace, \\
\uvel &= \h && \ton\R\times\partial\halfspace, \\
\uvel(\per+t, \cdot) &= \uvel(t, \cdot),
\end{pdeq}
\end{align}
in a half-space $\halfspace$ of dimension $n\geq 2$. Here, $\uvel\colon\R\times\halfspace\to\Rn$ denotes the velocity field and $\upres\colon\R\times\halfspace\to\R$ the pressure term. As customary in the formulation of time-periodic problems, the time-axis is taken to be the whole of $\R$. Variables in the time-space domain $\R\times\halfspace$ are denoted by $(t,x)$.
The time-period $\per>0$ shall remain fixed. Data  
$\f\colon\R\times\halfspace\to\Rn$, 
$\g\colon\R\times\halfspace\to\R$ and
$\h\colon\R\times\partial\halfspace\to\Rn$ 
that are also $\per$-time-periodic are considered.

The time-periodic Stokes equations play a fundamental role in a wide range of problems in fluid mechanics. Although 
comprehensive $\LR{p}$ estimates of maximal regularity type are available in the whole-space case \cite{KyedMaxReg14}, similar estimates in the more complicated half-space case were only established recently by \textsc{Maekawa} and \textsc{Sauer} \cite{MaekawaSauer17}. The analysis in \cite{MaekawaSauer17}, however, does not include inhomogeneous boundary data $\h\neq 0$. 
In the following, we shall establish maximal regularity $\LR{p}$ estimates that include the case $\h\neq 0$.
Such estimates are crucial in a number of applications. For example, the classical approach to free boundary problems in fluid-structure interaction relies heavily on maximal regularity frameworks that include inhomogeneous boundary data. When time-periodic driving forces are studied in such settings,
time-periodic Stokes equations appear in the linearization.

The nature of the Stokes problem does not allow the treatment of inhomogeneous boundary data by a simple ``lifting'' argument. 
Consequently, an extension of the results in \cite{MaekawaSauer17} to include the case $\h\neq 0$ is by no means trivial.
In the following, we employ a different approach than the reflection type argument used in \cite{MaekawaSauer17}. Instead, we use the Fourier-transform $\FT_{\torus\times\R^{n-1}}$ to reduce \eqref{SH} to an ordinary differential equation in the variable $x_n$. Here, $\torus$ denotes the torus $\R/\per\Z$. The $\LR{p}$ estimates are then established with arguments based on Fourier-multipliers and interpolation techniques. Although the main idea behind this approach is not new, indeed it has been applied successfully by various authors to investigate the initial-value Stokes problem, a number of non-trivial modifications are needed to adapt the arguments to the time-periodic case. Notably, the system \eqref{SH} has to be decomposed into a steady-state part and a so-called purely oscillatory part. Without this decomposition, it seems impossible to establish optimal $\LR{p}$ estimates. Whereas the estimates for the resulting steady-state problem are well-known and can be found in contemporary literature, the estimates for the purely oscillatory part in the following are new. 

It is convenient to formulate $\per$-time-periodic problems in a setting of function spaces where the torus $\torus\coloneqq\R/\per\Z$ is used as a time-axis.
Indeed, via lifting with the quotient map $\quotientmap\colon\R\to\torus$, 
$\per$-time-periodic functions are canonically identified as functions defined on $\torus$ and vice versa.
For such functions, we introduce the simple decomposition
\begin{align}\label{intro_projections}
\proj\uvel(x) \coloneqq \int_\torus \uvel(t,x)\,\dt = \iper\int_0^\per \uvel(t,x)\,\dt,\qquad
\projcompl\uvel(t,x) \coloneqq \uvel(t,x)-\proj\uvel(x)
\end{align} 
into a time-independent part $\proj\uvel$, and a part $\projcompl\uvel$ with vanishing time-average over the period. We shall refer to $\proj\uvel$ as the \emph{steady-state} part, and to $\projcompl\uvel$ as the \emph{purely oscillatory} part of $\uvel$. 

Equipped with the quotient topology, the time-space domain $\torus\times\R^n$ is a locally compact 
abelian group and therefore has a Fourier transform $\FT_{\torus\times\R^n}$ associated to it. Moreover, we may introduce the Schwartz-Bruhat space
$\SR(\torus\times\R^n)$ and its dual space $\TDR(\torus\times\R^n)$ of tempered distributions. 
Consequently, Bessel potential spaces with underlying time-space domain $\torus\times\R^n$ can be defined as subspaces of $\TDR(\torus\times\R^n)$ in a completely standard manner. Sobolev spaces are introduced as Bessel potential spaces with integer exponents, and Sobolev-Slobodecki\u{\i} spaces, \textit{i.e} Sobolev spaces with non-integer exponents, via real interpolation. 
The same scale of function spaces with respect to the half-space $\hs$ is obtained by restriction. 
In a setting of these function spaces (see Section \ref{pre} for the precise definitions and Remark \ref{PerVsTorusFunctionSpacesRemark} below), the main theorem of this article can be formulated as follows:

\begin{thm}[Main Theorem]\label{MainThm}
Let $q\in(1,\infty)$ and $n\geq 2$. For all
\begin{align}
\begin{aligned}\label{MainThm_Data}
&\f\in\LR{q}\bp{\torus; \LR{q}(\halfspace)}^n,\\
&\g\in\LR{q}\bp{\torus; \WSR{1}{q}(\halfspace)}\cap\WSR{1}{q}\bp{\torus; \WSRD{-1}{q}(\halfspace)},\\
&\h\in\WSR{1-\frac{1}{2q}}{q}\bp{\torus;\LR{q}(\ws)}^n\cap\LR{q}\bp{\torus;\WSR{2-\frac{1}{q}}{q}(\ws)}^n 
\end{aligned}
\end{align}
with
\begin{align}\label{MainThm_DataCompCond}
\begin{aligned}
&\h_{n}\in\WSR{1}{q}(\torus; \WSRD{-\frac{1}{q}}{q}(\ws))
\end{aligned}
\end{align}
there is a solution $(\uvel,\upres)$ to \eqref{SH} with 
\begin{align}
\begin{aligned}\label{MainThm_SolReg}
&\proj\uvel\in\WSRD{2}{q}(\hs)^n,\\
&\projcompl\uvel\in\projcompl\WSR{1}{q}\bp{\torus;\LR{q}(\hs)}^n\cap\projcompl\LR{q}\bp{\torus;\WSR{2}{q}(\hs)}^n,\\
&\upres\in\LR{q}\bp{\torus;\WSRD{1}{q}(\hs)},
\end{aligned}
\end{align}
which satisfies
\begin{multline}\label{MainThm_ProjEst}
\norm{\grad^2\proj\uvel}_{\LR{q}(\hs)} + \norm{\grad\proj\upres}_{\LR{q}(\hs)}\\
\leq \Ccn{C}
\bp{\norm{\proj \f}_{\LR{q}(\hs)} + \norm{\proj\g}_{\WSR{1}{q}(\hs)}+\norm{\proj\h}_{\WSR{2-\frac{1}{q}}{q}(\ws)}}
\end{multline}
and 
\begin{align}\label{MainThm_ProjComplEst}
\begin{aligned}
&\norm{\projcompl\uvel}_{\WSR{1}{q}\np{\torus;\LR{q}(\hs)}\cap\LR{q}\np{\torus;\WSR{2}{q}(\hs)}}
+ \norm{\grad\projcompl\upres}_{\LR{q}\np{\torus; \LR{q}(\hs)}} \\
&\qquad \leq \Ccn{C}\, 
\bp{\norm{\projcompl\f}_{\LR{q}\np{\torus;\LR{q}(\halfspace)}}+\norm{\projcompl\g}_{\LR{q}\np{\torus; \WSR{1}{q}(\halfspace)}\cap\WSR{1}{q}\np{\torus; \WSRD{-1}{q}(\halfspace)}}\\
&\qquad\qquad +\norm{\projcompl\h}_{\WSR{1-\frac{1}{2q}}{q}\np{\torus;\LR{q}(\ws)}\cap\LR{q}\np{\torus;\WSR{2-\frac{1}{q}}{q}(\ws)}}
+ \norm{\projcompl\h_n}_{\WSR{1}{q}\np{\torus; \WSRD{-\frac{1}{q}}{q}(\ws)}}\,
}
\end{aligned}
\end{align}
with $\Ccn{C}=\Ccn{C}\np{n,q,\per}$.
If $(\tuvel,\tupres)$ is another solution to \eqref{SH} in the class \eqref{MainThm_SolReg}, then 
$\projcompl\uvel=\projcompl\tuvel$, $\proj\uvel=\proj\tuvel+(a_1x_n,\ldots,a_{n-1}x_{n},0)$ for some vector $a\in\R^{n-1}$, and $\upres=\tupres + d(t)$
for some function $d$ that depends only on time.
\end{thm}

The two separated estimates \eqref{MainThm_ProjEst} and \eqref{MainThm_ProjComplEst} of different regularity type for the steady-state $\proj\uvel$ and
the {purely oscillatory} part $\projcompl\uvel$ of the solution, respectively, 
demonstrate the necessity of the decomposition. Observe that the purely oscillatory 
part $\projcompl\uvel$ of the solution is unique, whereas the steady-state part $\proj\uvel$ is not. The projections $\proj$ and $\projcompl$ 
decompose the time-periodic Stokes problem \eqref{SH} into a classical steady-state Stokes problem with respect to data $(\proj\f,\proj\g,\proj\h)$
and a time-periodic Stokes problem with respect to purely oscillatory data $(\projcompl\f,\projcompl\g,\projcompl\h)$, respectively. The first estimate 
\eqref{MainThm_ProjEst} is well-known for the former problem,  whence the main objective in the following will be to establish existence of a unique solution to the latter that satisfies \eqref{MainThm_ProjComplEst}.  

\begin{rem}\label{PerVsTorusFunctionSpacesRemark}
It is possible to avoid the analysis on the torus group $\torus\coloneqq\R/\per\Z$ and instead define the function spaces appearing in Theorem \ref{MainThm} as spaces of $\per$-periodic functions on $\R$.
For a Banach space $E$, let  
\begin{align*}
\CRiper\left(\R; E\right)\coloneqq \setc{f\in\CRi\left(\R; E\right)}{f(t+\per, x) = f(t, x) }
\end{align*}
denote the space of smooth $\per$-time-periodic $E$-valued functions. The Bochner-Lebesgue and Bochner-Sobolev spaces of time-periodic functions can then be introduced as
\begin{align*}
&\LRper{q}\left(\R; E\right)\coloneqq \closure{\CRiper\left(\R; E\right)}{\norm{\cdot}_{\LR{q}\left((0,\per); E\right)}},\\
&\WSRper{k}{q}\left(\R; E\right) \coloneqq \closure{\CRiper\left(\R; E\right)}{\norm{\cdot}_{\WSR{k}{q}\left((0,\per ); E\right)}}.
\end{align*}
Observe that the closures above are taken with respect to a time interval of period $\per$.
Time-periodic Sobolev-Slobodecki\u{\i} spaces can then be defined as real interpolation spaces in the usual way. 
Via the canonical quotient map $\quotientmap\colon\R\to\torus$, the spaces 
$\CRiper\bp{\R; E}$ and $\CRi\bp{\torus; E}$ are isometrically isomorphic with respect to 
the norms $\norm{\cdot}_{\WSR{k}{q}\left((0,\per ); E\right)}$ and $\norm{\cdot}_{\WSR{k}{q}\left(\torus; E\right)}$, respectively,
provided the Haar measure on $\torus$ is normalized appropriately. It follows that 
$\WSRper{s}{q}\left(\R; E\right)$ and $\WSR{s}{q}\left(\torus; E\right)$ are also isometrically isomorphic for all $s$.
In this manner, all the function spaces appearing in Theorem \ref{MainThm} have interpretations as $\per$-time-periodic Bochner spaces.
\end{rem}

\section{Preliminaries}\label{pre}

The objective of this section is to formalize the reformulation of \eqref{SH} in a setting where the time axis is replaced with the torus group
$\torus\coloneqq\R/\per\Z$. This includes definitions of the function spaces appearing in Theorem \ref{MainThm}.

\subsection{Topology, differentiable structure and Fourier transform}

We utilize $\grp$ as a time-space domain. Equipped with the quotient topology induced by the quotient mapping 
\begin{align*}
\pi :\R\times\R^n\to\grp, \ \pi\left(t, x\right) \coloneqq \left(\left[t\right], x\right),
\end{align*}
$\grp$ becomes a locally compact abelian group. We can identify $\grp$ with the domain $\left[0,\per\right)\times\R^n$ via the restriction $\pi\big|_{\left[0, \per\right)\times\Rn}$. The Haar measure $\dg$ on $\grp$ is the product of the Lebesgue measure on $\R^n$ and the Lebesgue measure on $\left[0,\per\right)$. We normalize $\dg$ so that
\begin{align*}
\int_\grp \uvel(g)\,\dg = \frac{1}{\per}\int_0^\per\int_{\Rn} \uvel(t, x)\,\dx\dt.
\end{align*}
There is a bijective correspondence between points $(k, \xi)\in\dualgrp$ and characters
$\chi: \grp\to\C$, $\chi\left(t, x\right)\coloneqq e^{ix\cdot\xi + ikt}$ on $\grp$. Consequently, we can identify the dual group of
$\grp$ with $\dualgrp$. The
compact-open topology on $\dualgrp$ reduces to the product of the Euclidean topology on $\Rn$ and the discrete topology on $\Zgrp$. The Haar measure on $\dualgrp$ is therefore the product of the counting measure on $\Zgrp$ and the Lebesgue measure on $\Rn$.

The spaces of smooth functions on $\grp$ and $\dualgrp$ are defined as
\begin{align}\label{DefOfSmoothFunctionsOnGrp}
\CRi(\grp) \coloneqq \setc{\uvel:\grp\ra\R}{\exists U\in\CRi\bp{\R\times\Rn}:\ U=\uvel\circ\pi}
\end{align}
and 
\begin{align*}
\CRi\Bp{\dualgrp} \coloneqq \setcL{\uvel\in\CR{}\Bp{\dualgrp}}{\forall k\in\Zgrp: \uvel(k, \cdot)\in\CRi(\Rn)},
\end{align*}
respectively.
Derivatives of a function $\uvel\in\CRi(\grp)$ are defined by 
\begin{align*}
\pt^\beta\px^\alpha\uvel \coloneqq \left[\pt^\beta\px^\alpha\left(\uvel\circ\pi\right)\right]\circ\Pi^{-1},
\end{align*} 
with $\Pi \coloneqq \pi\big|_{\left[0,\per\right)\times\Rn}$.
The notion of Schwartz spaces can be extended to locally compact abelian groups (see \cite{Bruhat} and \cite{kyedeiter_PIFBook}).
The so-called Schwartz-Bruhat space on $\grp$ is given by
\begin{align*}
\SR\np{\grp} \coloneqq \setc{\uvel\in\CRi(\grp)}{\forall (\alpha, \beta, \gamma)\in\N_0^n\times\N_0\times\N_0^n: \ \rho_{\alpha, \beta, \gamma}(\uvel)<\infty},
\end{align*}
where
\begin{align*}
\rho_{\alpha, \beta, \gamma}(\uvel) \coloneqq \sup_{(t,x)\in\grp}{\left|x^\gamma\pt^\beta\px^\alpha\uvel(t,x)\right|}.
\end{align*}
Equipped with the semi-norm topology of the family ${\{\rho_{\alpha, \beta, \gamma} | \left(\alpha, \beta, \gamma\right)\in\N_0^n\times\N_0\times\N_0^n \}}$, $\SR(\grp)$ becomes a topological vector space. The corresponding topological dual space $\TDR(\grp)$ equipped with the weak* topology is referred to as the space of tempered distributions on $\grp$. Distributional derivatives for a tempered distribution $\uvel$ are defined by duality as in the classical case.
The Schwartz-Bruhat space on $\dualgrp$ is 
\begin{multline*}
\SR\Bp{\dualgrp} \\\coloneqq \setcL{\uvel\in\CRi\Bp{\dualgrp}}{\forall (\alpha, \beta, \gamma)\in\N_0^n\times\N_0^n\times\N_0: \ \dualrho_{\alpha, \beta, \gamma}(\uvel)<\infty},
\end{multline*}
with the generic semi-norms
\begin{align*}
\dualrho_{\alpha, \beta, \gamma}(\uvel) \coloneqq \sup_{(k,\xi)\in\dualgrp}{\left|\xi^\alpha\partial_\xi^\beta k^\gamma\uvel(k,\xi)\right|}
\end{align*}
inducing the topology.

By $\FT_\grp$ we denote the Fourier transform associated to the locally compact abelian group $\grp$ equipped with the Haar measure introduced above:
\begin{align*}
&\FT_\grp:\SR\np{\grp}\ra\SR\bp{\dualgrp},\\
&\FT_\grp\nb{\uvel}(k,\xi)\coloneqq \ft{\uvel}(k,\xi) \coloneqq \frac{1}{\per}\int_0^\per\int_{\Rn} \uvel(t,x)\,\e^{-ix\cdot\xi-ik t}\,\dx\dt.
\end{align*}
Recall that $\FT_\grp:\SR(\grp)\ra\SR(\dualgrp)$ is a homeomorphism with inverse given by
\begin{align*}
&\iFT_\grp:\SR\bp{\dualgrp}\ra\SR\np{\grp},\\
&\iFT_\grp\nb{\wvel}(t,x)\coloneqq \sum_{k\in\Zgrp}\,\int_{\Rn} \wvel(k,\xi)\,\e^{ix\cdot\xi+ik t}\,\dxi,
\end{align*}
provided the Lebesgue measure $\dxi$ is normalized appropriately.
By duality, $\FT_\grp$ extends to a homeomorphism $\FT_\grp:\TDR(\grp)\ra\TDR(\dualgrp)$.

The Fourier symbol, with respect to $\FT_\grp$, of the projection $\proj$ introduced in \eqref{intro_projections} is the delta distribution $\delta_\Zgrp$, \textit{i.e.}, the
function $\delta_\Zgrp:\Zgrp\ra\C$ with $\delta_\Zgrp(0):=1$ and $\delta_\Zgrp(k):=0$ for $k\neq 0$. Via the symbol, the projections $\proj$ and $\projcompl$ extend to projections on $\TDR\np{\grp}$:
\begin{align}\label{SymbolsOfProjections}
\begin{aligned}
&\proj:\TDR\np{\grp}\ra\TDR\np{\grp},\quad \proj\uvel := \iFT_\grp\bb{\delta_\Zgrp\,\FT_\grp\nb{\uvel}},\\
&\projcompl:\TDR\np{\grp}\ra\TDR\np{\grp},\quad \projcompl\uvel := \iFT_\grp\bb{\np{1-\delta_\Zgrp}\,\FT_\grp\nb{\uvel}}.
\end{aligned}
\end{align}

At this point, we have introduced ample formalism to reformulate \eqref{SH} equivalently as a system of partial differential equations in the time-space domain $\grp$. Moreover, the Fourier transform $\FT_\grp$ enables us the investigate the systems in terms of Fourier-multipliers.
Due to the lack of a comprehensive $\LR{q}$-multiplier theory in the general group setting, we shall utilize a so-called
Transference Principle for this purpose.
The Transference Principle goes back to \textsc{De Leeuw} \cite{Leeuw1965}. The lemma below is due to \textsc{Edwards} and \textsc{Gaudry} \cite{EdwardsGaudry}.

\begin{thm}[\textsc{De Leeuw}, \textsc{Edwards} and \textsc{Gaudry}]\label{transference}
Let $\grpG$ and $\grpH$ be locally compact abelian groups. Moreover, let
$\Phi:\dualgrpG\ra\dualgrpH$ be a continuous homomorphism and $q\in [1,\infty]$. Assume that $m\in\LR{\infty}(\dualgrpH;\C)$ is a continuous $\LR{q}$-multiplier, i.e., there is a constant $C$ such that
\begin{align*}
\forall f\in\LR{2}\left(\grpH\right)\cap\LR{q}\left(\grpH\right):\quad \norm{\iFT_\grpH\nb{m\, \FT_\grpH\nb{f}}}_q\leq C\norm{f}_q.
\end{align*}
Then $m\circ\Phi\in\LR{\infty}(\dualgrpG; \C)$ is also an $\LR{q}$-multiplier with
\begin{align*}
\forall f\in\LR{2}\left(\grpG\right)\cap\LR{q}\left(\grpG\right):\quad \norm{\iFT_\grpG\nb{m\circ\Phi\, \FT_\grpG\nb{f}}}_q\leq C\norm{f}_q.
\end{align*}
\end{thm}
\begin{proof}
See \cite[Theorem B.2.1]{EdwardsGaudry}.
\end{proof}

\subsection{Function spaces}\label{FunktionSpacesSection}

Since our proof of the main theorem is based on Fourier-multipliers and interpolation theory, we find it convenient to 
introduce all the relevant Sobolev spaces via Bessel-Potential spaces. 
Classical (inhomogeneous) Bessel-Potential spaces  
are defined for $s\in\R$ and $q\in[1,\infty)$ by 
\begin{align*}
&\HSR{s}{q}\np{\R^n} \coloneqq  \setcl{u\in\TDR\np{\Rn}}{\iFT_{\Rn}\bb{(1+\snorm{\xi}^2)^\frac{s}{2}\FT_{\Rn}\nb{u}}\in\LR{q}\np{\Rn}},\\
&\norm{u}_{\HSR{s}{q}\np{\R^n}} \coloneqq \norm{u}_{s,q}\coloneqq\norm{\iFT_{\Rn}\bb{(1+\snorm{\xi}^2)^\frac{s}{2}\FT_{\Rn}\nb{u}}}_q.
\end{align*}
Classical Sobolev spaces on $\Rn$ are defined as Bessel-Potential spaces of integer order $k\in\Z$, and Sobolev spaces on 
the half-space $\halfspace$ via restriction:
\begin{align*}
\WSR{k}{q}\np{\R^n}\coloneqq\HSR{k}{q}\np{\Rn},\qquad\WSR{k}{q}\np{\halfspace}\coloneqq\setcl{u_{|\halfspace}}{u\in\WSR{k}{q}\np{\Rn}}.
\end{align*}
Observe that for negative-order spaces, \textit{i.e} when $k<0$, the Sobolev space $\WSR{k}{q}\np{\halfspace}$ coincides with the dual space $\bp{\WSR{-k}{q'}\np{\halfspace}}'$ and \emph{not} with the dual space $\np{\WSRN{-k}{q'}\np{\halfspace}}'$. 
In the following, it is essential that the former meaning of $\WSR{k}{q}\np{\halfspace}$ is used.   

Homogeneous Bessel-Potential spaces are defined in accordance with \cite{TriebelTheoryFunctionSpaces} 
by introducing the subspace
\begin{align*}
\SRh\np{\Rn}\coloneqq\setcl{\phi\in\SR\np{\Rn}}{\forall\alpha\in\N_0^n:\ \partial_x^\alpha\ft{\phi}\np{0}=0}
\end{align*}
of $\SR\np{\Rn}$, 
and for $s\in\R$ and $q\in[1,\infty)$ letting
\begin{align*}
&\HSRD{s}{q}\np{\Rn}\coloneqq\setcl{u\in\TDRh\np{\Rn}}{\iFT_{\Rn}\bb{\snorm{\xi}^s \FT_{\Rn}\nb{u}}\in\LR{q}\np{\Rn}},\\
&\norm{u}_{\HSRD{s}{q}\np{\Rn}}\coloneqq\norm{\iFT_{\Rn}\bb{\snorm{\xi}^s \FT_{\Rn}\nb{u}}}_q.
\end{align*}
Due to the lack of regularity of $\snorm{\xi}^s$ at the origin, the above definition of $\HSRD{s}{q}\np{\Rn}$ 
is not meaningful as a subspace of $\TDR\np{\Rn}$. Instead, $\HSRD{s}{q}\np{\Rn}$ is defined as a subspace of $\TDRh\np{\Rn}$.
As such, $\HSRD{s}{q}\np{\Rn}$ is clearly a Banach space.
As above, we define homogeneous Sobolev spaces on $\Rn$ as homogeneous Bessel-Potential spaces of integer order $k\in\Z$, and introduce homogeneous Sobolev spaces on 
the half-space $\halfspace$ via restriction:
\begin{align*}
\WSRD{k}{q}\np{\R^n}\coloneqq\HSRD{k}{q}\np{\Rn},\qquad\WSRD{k}{q}\np{\halfspace}\coloneqq\setcl{u_{|\halfspace}}{u\in\WSRD{k}{q}\np{\Rn}}.
\end{align*}
By the Hahn-Banach Theorem, any functional in $\TDRh\np{\Rn}$ can be extended to a tempered distribution in $\TDR\np{\Rn}$. 
If $s\leq 0$, the extension of an element in $\HSRD{s}{q}\np{\Rn}$ to $\TDR\np{\Rn}$ is unique. 
In the case $s>0$, one may verify that two extensions of an element in $\HSRD{s}{q}\np{\Rn}$ differ at most by addition of a polynomial of order strictly less than $s$. With this ambiguity in mind, one may consider $\HSRD{s}{q}\np{\Rn}$ as a normed ($s\leq 0$) and semi-normed ($s> 0$) subspace of $\TDR\np{\Rn}$.

Sobolev-Slobodecki\u{\i} spaces of both homogeneous and inhomogeneous type are defined via real interpolation in the usual way. For example, the spaces appearing in Theorem \ref{MainThm} are defined by
\begin{align*}
&\WSR{2-\frac{1}{q}}{q}(\ws)\coloneqq\bp{\LR{q}\np{\ws},\WSR{2}{q}\np\ws}_{1-\frac{1}{2q},q},\\
&\WSRD{-\frac{1}{q}}{q}(\ws)\ceq\bp{\WSRD{-1}{q}\np{\ws},\LR{q}\np\ws}_{1-\frac{1}{q},q}
\end{align*}
and equipped with the associated interpolation norms. 

Analogously, Bessel-Potential spaces with underlying time-space domain $\grp$ are introduced via the Fourier transform $\FT_{\grp}$ as
\begin{align*}
&\HSR{r}{q}\bp{\torus;\HSR{s}{q}\np\Rn}\ceq\\
&\qquad\setcl{u\in\TDR\np{\grp}}{\iFT_{\grp}\bb{\np{1+\snorm{k}^2}^{\frac{r}{2}}(1+\snorm{\xi}^2)^\frac{s}{2}\FT_{\grp}\nb{u}}\in\LR{q}\bp{\torus; \LR{q}\np\Rn}}
\end{align*}
equipped with the canonical norm. Again, we refer to Bessel-Potential spaces of integer order $k,l\in\N$ as Sobolev spaces:
\begin{align*}
\WSR{k}{q}\bp{\torus;\WSR{l}{q}\np\Rn}\ceq\HSR{k}{q}\bp{\torus;\HSR{l}{q}\np\Rn}.
\end{align*} 
Sobolev spaces on the time-space domain $\torus\times\hs$ are defined via restriction of the elements in the spaces above. 
In order to introduce homogeneous spaces, we let 
\begin{align*}
\SRh\np{\grp}\coloneqq\setcl{\phi\in\SR\np{\grp}}{\forall\alpha\in\N_0^n:\ \partial_x^\alpha\FT_{\R^n}\nb{\phi}\np{t,0}=0}
\end{align*}
and put
\begin{align*}
&\HSR{r}{q}\bp{\torus;\HSRD{s}{q}\np\Rn}\ceq\\
&\qquad\setcl{u\in\TDRh\np{\grp}}{\iFT_{\grp}\bb{\np{1+\snorm{k}^2}^{\frac{r}{2}}\snorm{\xi}^s\FT_{\grp}\nb{u}}\in\LR{q}\bp{\torus; \LR{q}\np\Rn}}.
\end{align*}
As above, we may consider $\HSR{r}{q}\bp{\torus;\HSRD{s}{q}\np\Rn}$ as a subspace of $\TDR\np{\grp}$ by extension.
Finally, Sobolev-Slobodecki\u{\i} spaces on the domain $\grp$ are defined via real interpolation. For example, 
\begin{align*}
\WSR{1-\frac{1}{2q}}{q}\bp{\torus;\LR{q}(\ws)}
\ceq \bp{\LR{q}\bp{\torus;\LR{q}\np\ws},\WSR{1}{q}\bp{\torus;\LR{q}\np\ws}}_{1-\frac{1}{2q},q}.
\end{align*}
In this way, all the function spaces appearing in Theorem \ref{MainThm} attain rigorous definitions. 
It is easy to verify that these definitions coincide with a classical interpretation as Bochner spaces of vector-valued functions defined on the torus 
$\torus$. 

\subsection{Interpolation}

Although the function spaces appearing in Theorem \ref{MainThm} can all be defined in terms of classical interpolation, our proof of the theorem
relies on a somewhat more refined scale of interpolation spaces. More specifically, it is based on anisotropic Besov spaces with underlying time-space domain $\grp$, which we shall show coincide with the function spaces obtained by real interpolation of the Bessel-Potential spaces introduced above.
Although this task is mainly technical and does not require any significantly new ideas, indeed we shall mimic the proofs of similar results for classical isotropic Besov spaces, these spaces and their interpolation properties are not part of contemporary literature and we shall therefore carry out the
identification here (even in slightly more generality than actually needed for the proof of the main theorem).
To this end, we fix an $m\in\N$ and introduce the parabolic length scale
\begin{align}\label{BS_LengthScale}
\parnorm{\eta}{\xi}\coloneqq(|\eta|^2+|\xi|^{4m})^{\frac{1}{4m}}\quad {\text{for } (\eta,\xi)\in\R\times\R^n}.
\end{align}
The anisotropic Besov spaces defined below pertain to time-periodic parabolic problems of order $2m$. In our analysis of the Stokes problem, we thus put
$m=1$. For simplicity, we omit $m$ in the notation for the function spaces below.  

The anisotropic Besov spaces shall be based on the following anisotropic partition of unit:

\begin{lem}\label{BS_PartitionOfUnity}
Let $m\in\N$ and $\parnorm{\eta}{\xi}$ be given by \eqref{BS_LengthScale}. There is a $\pu\in\CRci\np{\R\times\R^n}$ satisfying
\begin{align}
&\supp\pu = \setc{(\eta,\xi)}{2^{-1}\leq \parnorm{\eta}{\xi}\leq 2}  \label{BS_PartitionOfUnity_Support},\\
&\pu(\eta,\xi)>0 \quad\text{for}\quad 2^{-1}<\parnorm{\eta}{\xi}<2,\label{BS_PartitionOfUnity_Positivity}\\
&\sum_{l=-\infty}^\infty \pu(2^{-2ml}\eta,2^{-l}\xi)=1\quad \text{for}\quad \parnorm{\eta}{\xi}\neq 0.\label{BS_PartitionOfUnity_PartUnity}
\end{align}
\end{lem}

\begin{proof}
Let $\hf\in\CRi(\R)$ with $\supp\hf=\setc{y\in\R}{2^{-1}\leq\snorm{y}\leq2}$
and $h(y)>0$ for $2^{-1}<\snorm{y}<2$. Then $\f:\R\times\R^n\ra\R,\ \f(\eta,\xi)\coloneqq\hf\bp{\pex}$ satisfies
\eqref{BS_PartitionOfUnity_Support} and \eqref{BS_PartitionOfUnity_Positivity}. Moreover, $\f(2^{-2ml}\eta,2^{-l}\xi)\neq 0$ iff 
$2^{l-1}<\pex<2^{l+1}$. Thus $\f(2^{-2ml}\eta,2^{-l}\xi)\neq 0$ for at least one and at most two $l\in\Z$. Consequently,
\begin{align*}
\pu:\R\times\R^n\ra\R,\quad \pu(\eta,\xi)\coloneqq
\begin{pdeq}
&\frac{f(\eta,\xi)}{\sum_{l=-\infty}^\infty \f(2^{-2ml}\eta,2^{-l}\xi)} &&\tif\pex\neq 0,\\
&0 && \tif\pex=0
\end{pdeq}
\end{align*}
is well-defined. It is easy to verify that $\pu$ satisfies \eqref{BS_PartitionOfUnity_Support}--\eqref{BS_PartitionOfUnity_PartUnity}.
\end{proof}

\begin{defn}[Anisotropic Besov and Bessel-Potential Spaces]\label{BS_DefOfBesovSpace}
Let $\pu\in\CRci\np{\R\times\R^n}$ be as in Lemma \ref{BS_PartitionOfUnity}, $s\in\R$ and $p,q\in[1,\infty)$. We define anisotropic Besov spaces
\begin{align}\label{BS_DefOfBesovSpace_Defn}
\begin{aligned}
&\BSRcompl{s}{pq}\np{\grpx}\coloneqq\setc{\f\in\projcompl\TDR(\grpx)}{\norm{\f}_{\BSRcompl{s}{pq}}<\infty},\\
&\norm{\f}_{\BSRcompl{s}{pq}} \coloneqq \Bp{\sum_{l=0}^\infty \bp{2^{sl} \norm{\iFT_{\grpx}\bb{\pu(2^{-2ml}k,2^{-l}\xi) \FT_\grpx\nb{f} }}_p}^q}^{\frac{1}{q}},
\end{aligned}
\end{align} 
and anisotropic Bessel-Potential spaces
\begin{align}\label{BS_DefOfBesovSpace_DefnBessel}
\begin{aligned}
&\ABPSRcompl{s}{p}\np{\grpx}\coloneqq\setc{\f\in\projcompl\TDR(\grpx)}{\norm{\f}_{\ABPSRcompl{s}{p}}<\infty},\\
&\norm{\f}_{\ABPSRcompl{s}{p}} \coloneqq
\norm {\iFT_{\grp}\bb{\parnorm{k}{\xi}^s\FT_{\grp}\nb{\f}}}_p.
\end{aligned}
\end{align}
\end{defn}
Observe that $\BSRcompl{s}{pq}\np{\grpx}$ and $\ABPSRcompl{s}{p}\np{\grpx}$ are defined as subspaces of the purely oscillatory distributions $\projcompl\TDR(\grpx)$ rather than
$\TDR(\grpx)$. Recalling \eqref{SymbolsOfProjections}, it is easy to verify that $\norm{\cdot}_{\BSRcompl{s}{pq}}$
and $\norm{\cdot}_{\ABPSRcompl{s}{p}}$ are therefore norms (rather than mere semi-norms), and $\BSRcompl{s}{pq}\np{\grpx}$ and $\ABPSRcompl{s}{p}\np{\grpx}$
Banach spaces. 
As in the case of classical (isotropic) spaces, real interpolation of anisotropic Bessel-potential spaces yields anisotropic Besov spaces:

\begin{lem}\label{BS_InterpolationLem}
Let $p,q\in(1,\infty)$, $\theta\in(0,1)$, $s_0,s_1\in\R$ and $s\coloneqq(1-\theta)s_0+\theta s_1$.
If $s_0\neq s_1$, then 
$\bp{\ABPSRcompl{s_0}{p}\np\grpx,\ABPSRcompl{s_1}{p}\np\grpx}_{\theta,q}=\BSRcompl{s}{pq}\np\grpx$ with equivalent norms. 
\end{lem}

\begin{proof}
For $l\in\N_0$ and $r\in\R$ let 
\begin{align*}
\mult^r_l:\R\times\R^n\ra\C,\quad\mult^r_l(\eta,\xi)\coloneqq \pu\bp{2^{-2ml}\eta,2^{-l}\xi}\pex^{-r}.
\end{align*}
We claim that $\mult^r_l$ is an $\LR{p}\np{\R; \LR{p}\np\Rn}$-multiplier, which we verify by showing that $\mult^r_l$ meets the condition
of Marcinkiewicz's multiplier theorem (see for example \cite[Corollary 6.2.5]{Grafakos}).  
For this purpose, we utilize only 
\begin{align}
&\supp\pu\bp{2^{-2ml}\cdot,2^{-l}\cdot}=\setc{(\eta,\xi)\in\R\times\R^n}{2^{l-1}\leq\pex\leq2^{l+1}},\label{BS_InterpolationThm_SuppProperty}
\end{align}
and that $g(\eta,\xi)\coloneqq\pex^{-r}$ is parabolically $\np{-r}$-homogeneous, that is,
\begin{align}
&\forall\lambda>0:\quad \g(\eta,\xi)=\lambda^{-r} g(\lambda^{-2m}\eta,\lambda^{-1}\xi).\label{BS_InterpolationThm_HomoProperty}
\end{align}
From \eqref{BS_InterpolationThm_SuppProperty} we immediately obtain
$\norm{\mult^r_l}_\infty \leq \Ccn{C} \norm{\pu}_\infty 2^{-lr}$, 
with $\Ccnlast{C}$ independent on $l$. By \eqref{BS_InterpolationThm_HomoProperty}, we further observe that 
\begin{align*}
\eta\,\partial_\eta\mult^r_l(\eta,\xi) &= 2^{-2ml}\eta\, \partial_\eta\pu\bp{2^{-2ml}\eta,2^{-l}\xi}\, \g(\eta,\xi) \\
&\quad + \pu\bp{2^{-2ml}\eta,2^{-l}\xi}\,\lambda^{-r}\, \partial_\eta g(\lambda^{-2m}\eta,\lambda^{-1}\xi)\,\lambda^{-2m}\eta.
\end{align*}
Choosing $\lambda\coloneqq\pex$ and recalling \eqref{BS_InterpolationThm_SuppProperty}, we thus deduce
$\norm{\eta\,\partial_\eta\mult^r_l}_\infty \leq \Ccn{C} \norm{\pu}_\infty 2^{-lr}$, 
with $\Ccnlast{C}$ independent on $l$. Similarly, we obtain 
\begin{align*}
\sum_{\alpha\in\set{0,1}^{n+1}} 
\norm{\xi_1^{\alpha_1}\cdots\xi_n^{\alpha_n}\eta^{\alpha_{n+1}}
\partial_{\xi_1}^{\alpha_1}\cdots\partial_{\xi_n}^{\alpha_n}\partial_{\eta}^{\alpha_{n+1}} \mult^r_l}_\infty\leq \Ccn{C} \norm{\pu}_\infty 2^{-lr} 
\end{align*}
with $\Ccnlast{C}$ independent on $l$.
It follows from Marcinkiewicz's multiplier theorem that $\mult^r_l$ is an $\LR{p}(\R; \LR{p}(\R^n))$-multiplier. 
Consequently, the Transference Principle (Theorem \ref{transference}) implies that $\multrestriction{{\mult^r_l}}{\Zgrp\times\R^n}$ 
is an $\LR{p}(\torus; \LR{p}(\Rn))$-multiplier with
\begin{align*}
\normL{\phi\mapsto\iFT_{\grpx}\bb{\mult^r_l(k,\xi)\FT_\grpx\nb{\phi}}}_{\linearmaps\np{\LR{p}\np{\torus; \LR{p}\np\Rn},\LR{p}\np{\torus; \LR{p}\np\Rn}}}<\Ccn{C} \norm{\pu}_\infty 2^{-lr}.
\end{align*}
Let $\f\in\bp{\ABPSRcompl{s_0}{p}\np\grpx,\ABPSRcompl{s_1}{p}\np\grpx}_{\theta,q}$. Consider a decomposition $f=f_0+f_1$ with 
$\f_0\in\ABPSRcompl{s_0}{p}\np\grpx$ and $\f_1\in\ABPSRcompl{s_1}{p}\np\grpx$. We deduce
\begin{align*}
&\norm{\iFT_{\grpx}\bb{\pu(2^{-2ml}k,2^{-l}\xi) \FT_\grpx\nb{f} }}_p \\
&\quad\leq \norm{\iFT_{\grpx}\bb{\mult^{s_0}_l\FT_\grpx\bb{\iFT_{\grpx}\bb{\parnorm{k}{\xi}^{s_0}\FT_\grpx\nb{\f_0}}}}}_p\\
&\quad\quad + \norm{\iFT_{\grpx}\bb{\mult^{s_1}_l\FT_\grpx\bb{\iFT_{\grpx}\bb{\parnorm{k}{\xi}^{s_1}\FT_\grpx\nb{\f_1}}}}}_p\\
&\quad\leq \Ccn{C}\bp{2^{-ls_0}\norm{f_0}_{\ABPSRcompl{s_0}{p}} + 2^{-ls_1}\norm{f_1}_{\ABPSRcompl{s_1}{p}} }\\
&\quad\leq \Ccn{C}2^{-ls_0}\bp{\norm{f_0}_{\ABPSRcompl{s_0}{p}} + 2^{l(s_0-s_1)}\norm{f_1}_{\ABPSRcompl{s_1}{p}} }.
\end{align*}
We now employ the $K$-method (see for example \cite[Chapter 3.1]{BL76}) 
to characterize the interpolation space $\bp{\ABPSRcompl{s_0}{p}\np\grpx,\ABPSRcompl{s_1}{p}\np\grpx}_{\theta,q}$. Taking infimum over all decompositions $f_0,f_1$ in the inequality above, we find that 
\begin{align*}
\norm{\iFT_{\grpx}\bb{\pu(2^{-2ml}k,2^{-l}\xi) \FT_\grpx\nb{f} }}_p \leq \Ccn{C}\, 2^{-ls_0}\,K\bp{2^{l(s_0-s_1)},\f,\ABPSRcompl{s_0}{p},\ABPSRcompl{s_1}{p}}, 
\end{align*}
which implies
\begin{align*}
\norm{\f}_{\BSRcompl{s}{pq}} \leq \Ccn{C} \Bp{\sum_{l=0}^\infty\bp{2^{\theta l(s_1-s_0)}\,K\np{2^{l(s_0-s_1)},\f,\ABPSRcompl{s_0}{p},\ABPSRcompl{s_1}{p}}}^q}^{\frac{1}{q}}
\leq \Ccn{C}\,\norm{\f}_{\bp{\ABPSRcompl{s_0}{p},\ABPSRcompl{s_1}{p}}_{\theta,q}},
\end{align*}
where the last inequality above is valid since $s_0\neq s_1$.

Now consider $\f\in\BSRcompl{s}{pq}\np\grpx$. Let $l\in\N_0$. Choose $\psi\in\CRci(\R\times\R^n)$ with $\psi(\eta,\xi)=1$ for $2^{-1}\leq\pex\leq2$ and 
$\supp\psi=\setc{(\eta,\xi)\in\R\times\R^n}{4^{-1}\leq\pex\leq 4}$. 
Using the same technique as above, this time utilizing the multiplier 
\begin{align*}
\tmult^r_l:\R\times\R^n\ra\C,\quad\tmult^r_l(\eta,\xi)\coloneqq \psi\bp{2^{-2ml}\eta,2^{-l}\xi}\pex^{-r},
\end{align*}
we can estimate
\begin{align*}
&\norm{\iFT_{\grpx}\bb{\pu(2^{-2ml}k,2^{-l}\xi) \FT_\grpx\nb{f} }}_{\ABPSRcompl{s_1}{p}}\\
&\qquad=\norm{\iFT_{\grpx}\bb{\psi(2^{-2ml}k,2^{-l}\xi)\, \pu(2^{-2ml}k,2^{-l}\xi)\, \FT_\grpx\nb{f} }}_{\ABPSRcompl{s_1}{p}}\\ 
&\qquad \leq \Ccn{C}\,2^{ls_1}\norm{\iFT_{\grpx}\bb{\pu(2^{-2ml}k,2^{-l}\xi) \FT_\grpx\nb{f} }}_{p},
\end{align*}
and similarly
\begin{align*}
\norm{\iFT_{\grpx}\bb{\pu(2^{-2ml}k,2^{-l}\xi) \FT_\grpx\nb{f} }}_{\ABPSRcompl{s_0}{p}} \leq \Ccn{C}\,2^{ls_0}\norm{\iFT_{\grpx}\bb{\pu(2^{-2ml}k,2^{-l}\xi) \FT_\grpx\nb{f} }}_{p}.
\end{align*}
We thus obtain
\begin{align*}
&2^{-l\theta(s_0-s_1)}{2^{l(s_0-s_1)} \norm{\iFT_{\grpx}\bb{\pu(2^{-2ml}k,2^{-l}\xi) \FT_\grpx\nb{f} }}_{\ABPSRcompl{s_1}{p}}} \\
&\qquad=2^{ls} 2^{-ls_1} \norm{\iFT_{\grpx}\bb{\pu(2^{-2ml}k,2^{-l}\xi) \FT_\grpx\nb{f} }}_{\ABPSRcompl{s_1}{p}}  \\
&\qquad\leq \Ccn{C} 2^{ls} \norm{\iFT_{\grpx}\bb{\pu(2^{-2ml}k,2^{-l}\xi) \FT_\grpx\nb{f} }}_{p}  
\end{align*}
and
\begin{align*}
&2^{-l\theta(s_0-s_1)}{ \norm{\iFT_{\grpx}\bb{\pu(2^{-2ml}k,2^{-l}\xi) \FT_\grpx\nb{f} }}_{\ABPSRcompl{s_0}{p}}} \\
&\qquad=2^{ls} 2^{-ls_0} \norm{\iFT_{\grpx}\bb{\pu(2^{-2ml}k,2^{-l}\xi) \FT_\grpx\nb{f} }}_{\ABPSRcompl{s_0}{p}}  \\
&\qquad\leq \Ccn{C} 2^{ls} \norm{\iFT_{\grpx}\bb{\pu(2^{-2ml}k,2^{-l}\xi) \FT_\grpx\nb{f} }}_{p}.  
\end{align*}
We now employ the $J$-method (see for example \cite[Chapter 3.2]{BL76}) to characterize
the interpolation space $\bp{\ABPSRcompl{s_0}{p}\np\grpx,\ABPSRcompl{s_1}{p}\np\grpx}_{\theta,q}$. By the last two estimates above, we see that 
\begin{align*}
&2^{-l\theta(s_0-s_1)}\,J\bp{2^{l(s_0-s_1)},\iFT_{\grpx}\bb{\pu(2^{-2ml}k,2^{-l}\xi) \FT_\grpx\nb{f} }}\\ 
&\qquad\leq \Ccn{C}\,2^{ls} \norm{\iFT_{\grpx}\bb{\pu(2^{-2ml}k,2^{-l}\xi) \FT_\grpx\nb{f}}}_p.
\end{align*} 
Since $\projcompl\f=\f$, we find that $f=\sum_{l=0}^\infty \iFT_{\grpx}\bb{\pu(2^{-2ml}k,2^{-l}\xi) \FT_\grpx\nb{f}}$ with 
convergence in the space $\ABPSRcompl{s_0}{p}\np\grpx+\ABPSRcompl{s_1}{p}\np\grpx$. Recalling that $s_0\neq s_1$, we thus conclude that 
\begin{align*}
\norm{\f}_{\bp{\ABPSRcompl{s_0}{p},\ABPSRcompl{s_1}{p}}_{\theta,q}} \leq \Ccn{C}\,\norm{\f}_{\BSRcompl{s}{pq}},
\end{align*}
and thereby the lemma.
\end{proof}

\section{Proof of Main Theorem}

Utilizing the formalism introduced in Section \ref{pre}, we can equivalently reformulate \eqref{SH} in a setting where the time axis $\R$ is replaced with the torus $\torus\coloneqq\R/\per\Z$. In this setting, the periodicity condition is no longer needed and we obtain the equivalent problem
\begin{align}\label{SHTorus}
\begin{pdeq}
\pt\uvel - \Delta\uvel + \grad \upres &= \f && \tin\torus\times\halfspace, \\
\Div\uvel &= \g && \tin\torus\times\halfspace, \\
\uvel &= \h && \ton\torus\times\partial\halfspace.
\end{pdeq}
\end{align}
In order to investigate \eqref{SHTorus}, we employ the projections $\proj$ and $\projcompl$ to decompose the problem into a steady-state and
a so-called \emph{purely oscillatory} problem. More specifically, we observe that $\np{\uvel,\upres}$ is a solution to \eqref{SHTorus} if and only if $(\vvel,\vpres)\coloneqq(\proj\uvel,\proj\upres)$ is a solution to the steady-state problem  
\begin{align}\label{SHGS}
\begin{pdeq}
- \Delta\vvel + \grad\vpres &= \proj\f && \tin\halfspace, \\
\Div\vvel &= \proj\g && \tin\halfspace, \\
\vvel &= \proj\h && \ton\partial\halfspace
\end{pdeq}
\end{align}
and $\np{\wvel,\wpres}\coloneqq\np{\projcompl\uvel,\projcompl\upres}$ is a solution to 
\begin{align}\label{SHGP}
\begin{pdeq}
\pt\wvel - \Delta\wvel + \grad\wpres &= \projcompl\f && \tin\torus\times\halfspace, \\
\Div\wvel &= \projcompl\g && \tin\torus\times\halfspace, \\
\wvel &=  \projcompl\h && \ton\torus\times\partial\halfspace.
\end{pdeq}
\end{align}
The steady-state problem \eqref{SHGS} is a classical Stokes problem, for which a comprehensive theory is available. We therefore focus on the
purely oscillatory problem \eqref{SHGP}, which only differs from \eqref{SH} by having purely oscillatory data.
We start with the case of non-homogeneous boundary values. 

\begin{prop}\label{PurelyOscProblem_HomoBrdData}
Let $q\in (1, \infty)$ and $n\geq 2$. For any vector field $\rhsh$ with
\begin{align}\label{PurelyOscProblem_HomoBrdData_DataReg}
\begin{aligned}
&\rhsh\in\projcompl\WSR{1-\frac{1}{2q}}{q}\bp{\torus;\LR{q}(\ws)}^n\cap\projcompl\LR{q}\bp{\torus;\WSR{2-\frac{1}{q}}{q}(\ws)}^n,\\ 
&\rhsh_{n}\in\projcompl\WSR{1}{q}(\torus; \WSRD{-\frac{1}{q}}{q}(\ws))
\end{aligned}
\end{align}
there is a solution 
\begin{align}\label{PurelyOscProblem_HomoBrdData_SolReg}
\begin{aligned}
&\uvel\in\projcompl\WSR{1}{q}\bp{\torus;\LR{q}(\hs)}^n\cap\projcompl\LR{q}\bp{\torus;\WSR{2}{q}(\hs)}^n,\\
&\upres\in\projcompl\LR{q}\bp{\torus;\WSRD{1}{q}(\hs)}
\end{aligned}
\end{align}
to 
\begin{align}\label{PurelyOscProblem_HomoBrdData_Eq}
\begin{pdeq}
\pt\uvel - \Delta\uvel + \grad \upres &= 0 && \tin\torus\times\halfspace, \\
\Div\uvel &= 0 && \tin\torus\times\halfspace, \\
\uvel &= \rhsh && \ton\torus\times\partial\halfspace
\end{pdeq}
\end{align}
that satisfies 
\begin{align}
\begin{aligned}\label{PurelyOscProblem_HomoBrdData_Est}
&\norm{\uvel}_{\WSR{1}{q}\np{\torus;\LR{q}(\hs)}\cap\LR{q}\np{\torus;\WSR{2}{q}(\hs)}}
+ \norm{\grad\upres}_{\LR{q}\np{\torus;\LR{q}(\hs)}}\\
&\qquad \leq \Ccn{C}\, 
\bp{\norm{\rhsh}_{\WSR{1-\frac{1}{2q}}{q}\np{\torus;\LR{q}(\ws)}\cap\LR{q}\np{\torus;\WSR{2-\frac{1}{q}}{q}(\ws)}}
+\norm{\rhsh_n}_{\WSR{1}{q}\np{\torus; \WSRD{-\frac{1}{q}}{q}(\ws)}}
}
\end{aligned}
\end{align}
with $\Ccn{C}=\Ccn{C}(n,q,\per)$.
\end{prop}

\begin{proof}
We shall employ the Fourier transform $\FT_{\torus\times\ws}$ to transform \eqref{PurelyOscProblem_HomoBrdData_Eq} into a system of ODEs. 
For this purpose, we denote by $(k,\xi)\in \Zgrp\times\ws$ the coordinates in the dual group of $\torus\times\ws$. 
Letting 
$\uvelp\ceq\uvel^\prime \coloneqq (\uvel_1,\ldots,\uvel_{n-1})$, $\uveln\coloneqq\uvel_n$, 
$\ft{\uvelp}\ceq\FT_{\torus\times\ws}\bb{\uvelp}$, and $\ft{\uveln}\coloneqq\FT_{\torus\times\ws}\bb{\uvel_n}$ in \eqref{PurelyOscProblem_HomoBrdData_Eq}, we obtain an equivalent formulation of the system as a family of ODEs. More precisely, \eqref{PurelyOscProblem_HomoBrdData_Eq} is equivalent
to the following ODE being satisfied for each (fixed) $(k,\xi)\in\Zgrp\times\ws$:
\begin{align}\label{ODE}
\begin{pdeq}
ik\ft{\uvelp}(x_n) + \snorm{\xi}^2\ft{\uvelp}(x_n) - \partial_{x_n}^2\ft{\uvelp}(x_n) + i\xi\ft{\upres}(x_n) &= 0 && \tin\R_+, \\
ik\ft{\uveln}(x_n) + \snorm{\xi}^2\ft{\uveln}(x_n) - \partial_{x_n}^2\ft{\uveln}(x_n) + \partial_{x_n}\ft{\upres}(x_n) &= 0 && \tin\R_+, \\
i\xi\cdot\ft{\uvelp}(x_n) + \partial_{x_n}\ft{\uveln}(x_n) &= 0 && \tin\R_+, \\
(\ft{\uvelp}(0), \ft{\uveln}(0)) &= (\ft{\rhshp}, \ft{\rhshn}). 
\end{pdeq}
\end{align}
To solve the ODE, we first consider the case $k\neq 0$.
Taking divergence on both sides in $\eqrefsub{PurelyOscProblem_HomoBrdData_Eq}{1}$ and utilizing that $\Div\uvel=0$, we find that 
$\Delta\upres = 0$ and thus $-\snorm{\xi}^2\ft{\upres}(x_n) + \partial_{x_n}^2\ft{\upres}(x_n) = 0$ in $\R_+$. Consequently, 
\begin{align}\label{PurelyOscProblem_tfupresFormula}
\ft{\upres}(x_n) = \qdf(k,\xi) e^{-\snorm{\xi}x_n}\quad\tin\R_+
\end{align}
for some function $\qdf:\Zgrp\times\ws\ra\C$. 
Inserting \eqref{PurelyOscProblem_tfupresFormula} into \eqref{ODE}, we find that
\begin{align*}
\partial_{x_n}^2\ft{\uvelp} &= (ik + \snorm{\xi}^2)\ft{\uvelp} + i\xi \qdf e^{-\snorm{\xi}x_n}, \\
\partial_{x_n}^2\ft{\uveln} &= (ik + \snorm{\xi}^2)\ft{\uveln} - \snorm{\xi} \qdf e^{-\snorm{\xi}x_n}.
\end{align*}
Since $k\neq 0$, the resolution hereof yields 
\begin{align}
\ft{\uvelp}(x_n) &= -\frac{\xi\qdf(k, \xi)}{k} e^{-\snorm{\xi}x_n} + \alpha(k, \xi) e^{-\sqrt{\snorm{\xi}^2 + ik} \, x_n}, \\
\ft{\uveln}(x_n) &= \frac{\snorm{\xi}\qdf(k, \xi)}{ik} e^{-\snorm{\xi}x_n} + \beta(k, \xi) e^{-\sqrt{\snorm{\xi}^2 + ik} \, x_n},
\end{align}
for some functions $\alpha,\beta:\Zgrp\times\ws\ra\C$. 
Utilizing 
$\eqrefsub{ODE}{3}$ and the boundary conditions $\eqrefsub{ODE}{5}$, we deduce
\begin{align*}
&\alpha = \ft{\rhshp} + \frac{\xi\qdf}{k}, \qquad \beta = \ft{\rhshn} - \frac{\snorm{\xi}\qdf}{ik}
\end{align*}
and
\begin{align}\label{PurelyOscProblem_HomoBrdData_SolFormulaq0}
&\qdf = -i\left(\snorm{\xi} + \sqrt{\snorm{\xi}^2 + ik}\right)\frac{\xi}{\snorm{\xi}}\cdot\ft{\rhshp} + \sqrt{\snorm{\xi}^2+ik} \, \ft{\rhshn} + 
\snorm{\xi}\ft{\rhshn} + \frac{ik}{\snorm{\xi}}\ft{\rhshn}.
\end{align}
By \eqref{PurelyOscProblem_tfupresFormula}--\eqref{PurelyOscProblem_HomoBrdData_SolFormulaq0} a solution to \eqref{ODE} is identified in the case
$k\neq 0$. Since $\rhsh$ is purely oscillatory, we have $\ft{\rhshp}(0,\xi)=\ft{\rhshn}(0,\xi)=0$, whence $(\uvelp,\uveln,\upres)\coloneqq(0,0,0)$ solves \eqref{ODE} in the case $k=0$.
We thus obtain a formula for the solution to \eqref{PurelyOscProblem_HomoBrdData_Eq}:  
\begin{align}\label{PurelyOscProblem_HomoBrdData_SolFormula}
\begin{aligned}
\uvelp &= \iFT_{\torus\times\ws}\Bb{-\frac{\xi\qdf}{k} e^{-\snorm{\xi}x_n} + \bp{\ft{\rhshp} + \frac{\xi\qdf}{k}} \e^{-\sqrt{\snorm{\xi}^2 + ik} \, x_n}}, \\
\uveln &= \iFT_{\torus\times\ws}\Bb{\frac{\snorm{\xi}\qdf}{ik} e^{-\snorm{\xi}x_n} + \bp{\ft{\rhshn} - \frac{\snorm{\xi}\qdf}{ik}} 
e^{-\sqrt{\snorm{\xi}^2 + ik} \, x_n}}, \\
\upres &= \iFT_{\torus\times\ws}\bb{\qdf e^{-\snorm{\xi}x_n}}.
\end{aligned}
\end{align}
Formally at least, $(\uvelp,\uveln,\upres)$ as defined above is a solution to \eqref{PurelyOscProblem_HomoBrdData_Eq}. It remains to show that 
this solution is well-defined in the class \eqref{PurelyOscProblem_HomoBrdData_SolReg} for data in the class
\eqref{PurelyOscProblem_HomoBrdData_DataReg}. We start by considering data $\rhsh\in\projcompl\SRh\np{\torus\times\ws}^n$. The space
$\SRh\np{\torus\times\ws}$ is dense in 
\begin{align*}
\WSR{1-\frac{1}{2q}}{q}\bp{\torus;\LR{q}(\ws)}\cap\LR{q}\bp{\torus;\WSR{2-\frac{1}{q}}{q}(\ws)} \cap \WSR{1}{q}\bp{\torus; \WSRD{-\frac{1}{q}}{q}(\ws)},
\end{align*}
which is not a trivial assertion since it entails the construction of an approximating sequence that converges simultaneously in Sobolev spaces of positive order and in homogeneous Sobolev spaces of negative order. Nevertheless, it can be shown by a standard ``cut-off'' and mollifier technique;
see \cite[proof of Theorem 2.3.3 and Theorem 5.1.5]{TriebelTheoryFunctionSpaces}. 
Consequently, $\projcompl\SRh\np{\torus\times\ws}^n$ is dense in the class \eqref{PurelyOscProblem_HomoBrdData_DataReg}.
Clearly, for purely oscillatory data $\rhsh$ the solution 
given by \eqref{PurelyOscProblem_HomoBrdData_SolFormula} is also purely oscillatory.
If we can therefore show \eqref{PurelyOscProblem_HomoBrdData_Est} for arbitrary $\rhsh\in\projcompl\SRh\np{\torus\times\ws}^n$, the claim of the proposition will follow by a density argument. 

We first examine the pressure term $\upres$ (more specifically $\grad\upres$). The terms in \eqref{PurelyOscProblem_HomoBrdData_SolFormulaq0} have different 
order of regularity, so we decompose $\qdf=\qdfone+\qdftwo$ by 
\begin{align}\label{PurelyOscProblem_HomoBrdData_DefOfqdfs}
\begin{aligned}
&\qdfone(k,\xi) \ceq -i\left(\snorm{\xi} + \sqrt{\snorm{\xi}^2 + ik}\right)\frac{\xi}{\snorm{\xi}}\cdot\ft{\rhshp} + \sqrt{\snorm{\xi}^2+ik} \, \ft{\rhshn} + \snorm{\xi}\ft{\rhshn},\\
&\qdftwo(k,\xi) \ceq \frac{ik}{\snorm{\xi}}\ft{\rhshn},
\end{aligned}
\end{align}
and introduce the operators
\begin{align}\label{PurelyOscProblem_HomoBrdData_DefOfGoodOpr}
\begin{aligned}
&\qoprgood: \SRh\np{\torus\times\ws}^n \ra \SR\np{\torus\times\hs}^n, \\
&\qoprgood(\rhsh)\coloneqq \iFT_{\torus\times\ws}\Bb{\xi\qdfone(k,\xi)e^{-\snorm{\xi}x_n}}
\end{aligned}
\end{align}
and 
\begin{align}\label{PurelyOscProblem_HomoBrdData_DefOfBadOpr}
\begin{aligned}
&\qoprbad: \SRh\np{\torus\times\ws} \ra \SR\np{\torus\times\hs}^n, \\
&\qoprbad(\rhshn)\coloneqq \iFT_{\torus\times\ws}\Bb{\xi\qdftwo(k,\xi)e^{-\snorm{\xi}x_n}}.
\end{aligned}
\end{align}
For $m\in\N_0$, we observe for any $x_n>0$  that
the symbol $\xi\ra\np{\snorm{\xi}x_n}^m\e^{-\snorm{\xi}x_n}$ is an $\LR{q}\np{\ws}$-multiplier.
Specifically, one verifies that  
\begin{align*}
\sup_{x_n>0}\sup_{\varepsilon\in\{0, 1\}^{n-1}}\sup_{\xi\in\ws}\left|\xi_1^{\varepsilon_1}\cdots\xi_{n-1}^{\varepsilon_{n-1}}\partial_{\xi_1}^{\varepsilon_1}\cdots\partial_{\xi_{n-1}}^{\varepsilon_{n-1}}\bb{\np{\snorm{\xi}x_n}^m\e^{-\snorm{\xi}x_n}} \right|<\infty,
\end{align*}
whence it follows from the Marcinkiewicz Multiplier Theorem (see for example \cite[Corollary 6.2.5]{Grafakos}) that the Fourier-multiplier operator with  symbol 
$\xi\ra\np{\snorm{\xi}x_n}^m\e^{-\snorm{\xi}x_n}$ is a bounded operator on $\LR{q}\np\ws$ with operator norm independent on $x_n$, that is,
\begin{align}\label{PurelyOscProblem_HomoBrdData_MultiplierOprNormIndep}
\sup_{x_n>0}\,\normL{\phi\mapsto\FT_{\ws}\Bb{\np{\snorm{\xi}x_n}^m\e^{-\snorm{\xi}x_n}\FT_{\ws}\nb{\phi}}}_{\linearmaps\np{\LR{q}\np{\ws},\LR{q}\np{\ws}}}<\infty.
\end{align}
We can thus estimate 
\begin{align*}
&\norm{\qoprgood(\rhsh)}_{\LR{\infty}\np{\R_+;\LR{q}\np{\torus;\LR{q}(\ws)}}}
\leq \Ccn{C}\, \norm{\iFT_{\torus\times\ws}\bb{\xi\qdfone(k,\xi)}}_{\LR{q}\np{\torus;\LR{q}(\ws)}}\\
&\qquad\leq \Ccn{C}\, \Bp{\norm{\rhsh}_{\LR{q}\np{\torus;\HSR{2}{q}\np\ws}} + 
\normL{\iFT_{\torus\times\ws}\bb{\Mmultiplier(k,\xi)\,{\np{\snorm{\xi}^2+ik}\, \frac{\xi\otimes\xi}{\snorm{\xi}^2}\ft{\rhshp}}}}_{\LR{q}\np{\torus;\LR{q}(\ws)}}\\
&\qquad\qquad\ 
+\normL{\iFT_{\torus\times\ws}\bb{\Mmultiplier(k,\xi)\,{\np{\snorm{\xi}^2+ik}\, \frac{\xi}{\snorm{\xi}}\ft{\rhshn}}}}_{\LR{q}\np{\torus;\LR{q}(\ws)}}
},
\end{align*}
where
\begin{align*}
\Mmultiplier:\R\times\ws\ra\C,\quad
\Mmultiplier(\eta,\xi)\ceq \frac{\snorm{\xi}}{\sqrt{\snorm{\xi}^2+i\eta}}.
\end{align*}
Employing again the Marcinkiewicz Multiplier Theorem, we find that the symbol $\Mmultiplier$ is an $\LR{q}\np{\R;\LR{q}(\ws)}$-multiplier. An application of the Transference Principle (Theorem \ref{transference}) therefore implies that the restriction $\MmultiplierRestriction$ is an $\LR{q}\np{\torus;\LR{q}(\ws)}$-multiplier.
We thus conclude 
\begin{align}\label{PurelyOscProblem_HomoBrdData_GoodOprInterpolationEst1}
\begin{aligned}
&\norm{\qoprgood(\rhsh)}_{\LR{\infty}\np{\R_+;\LR{q}\np{\torus;\LR{q}(\ws)}}}
\leq \Ccn{C}\, {\norm{\rhsh}_{\LR{q}\np{\torus;\HSR{2}{q}\np\ws}\cap\HSR{1}{q}\np{\torus;\LR{q}\np{\ws}}}}.
\end{aligned}
\end{align}
This estimate shall serve as an interpolation endpoint. To obtain the opposite endpoint, we again employ \eqref{PurelyOscProblem_HomoBrdData_MultiplierOprNormIndep} to estimate
\begin{align*}
\sup_{x_n>0}\,\norm{x_n \qoprgood\np{\rhsh}}_{\LR{q}\np{\torus;\LR{q}(\ws)}}\leq \Ccn{C}\,\norm{\qdfone}_{\LR{q}\np{\torus;\LR{q}(\ws)}},
\end{align*}
which implies
\begin{align*}
\norm{\qoprgood\np{\rhsh}}_{\LR{1,\infty}\np{\R_+;\LR{q}\np{\torus;\LR{q}(\ws)}}}
&= \normL{\frac{1}{x_n} \norm{x_n \qoprgood\np{\rhsh}}_{\LR{q}\np{\torus;\LR{q}(\ws)}}}_{\LR{1,\infty}\np{\R_+}} \\
&\leq \Ccn{C}\,\norm{\qdfone}_{\LR{q}\np{\torus;\LR{q}(\ws)}}.
\end{align*}
Recalling \eqref{PurelyOscProblem_HomoBrdData_DefOfqdfs}, we estimate
\begin{align*}
&\norm{\qdfone}_{\LR{q}\np{\torus;\LR{q}(\ws)}}\\
&\qquad \leq \Ccn{C}\,\Bp{\norm{\rhsh}_{\LR{q}\np{\torus;\HSR{1}{q}\np\ws}}
+\normL{\iFT_{\torus\times\ws}\bb{\MmultiplierNr{1}(k,\xi)\cdot{\np{\snorm{\xi}+\snorm{k}^\half}\ft{\rhshp}}}}_{\LR{q}\np{\torus;\LR{q}(\ws)}}\\
&\qquad\qquad\quad +\normL{\iFT_{\torus\times\ws}\bb{\MmultiplierNr{2}(k,\xi)\,{\np{\snorm{\xi}+\snorm{k}^\half}\ft{\rhshn}}}}_{\LR{q}\np{\torus;\LR{q}(\ws)}}}
\end{align*}
with
\begin{align*}
&\MmultiplierNr{1}:\R\times\ws\ra\C^{n-1},\quad
\MmultiplierNr{1}(\eta,\xi)\ceq \frac{\sqrt{\snorm{\xi}^2+i\eta}}{\snorm{\xi}+\snorm{\eta}^{\frac{1}{2}}}\frac{\xi}{\snorm{\xi}},\\
&\MmultiplierNr{2}:\R\times\ws\ra\C,\quad
\MmultiplierNr{2}(\eta,\xi)\ceq \frac{\sqrt{\snorm{\xi}^2+i\eta}}{{\snorm{\xi}+\snorm{\eta}^{\frac{1}{2}}}}.
\end{align*}
Again, one can utilize the Marcinkiewicz Multiplier Theorem to show that both $\MmultiplierNr{1}$ and $\MmultiplierNr{2}$ are
$\LR{q}\np{\R;\LR{q}(\ws)}$-multipliers, and subsequently obtain from
the Transference Principle (Theorem \ref{transference}) that their restrictions to $\Zgrp\times\ws$ 
are $\LR{q}\np{\torus;\LR{q}(\ws)}$-multipliers. Consequently, we find that
\begin{align}\label{PurelyOscProblem_HomoBrdData_GoodOprInterpolationEst2}
&\norm{\qoprgood\np{\rhsh}}_{\LR{1,\infty}\np{\R_+;\LR{q}\np{\torus;\LR{q}(\ws)}}}
\leq \Ccn{C}\, {\norm{\rhsh}_{\LR{q}\np{\torus;\HSR{1}{q}\np\ws}\cap\HSR{\frac{1}{2}}{q}\np{\torus;\LR{q}\np{\ws}}}}.
\end{align}
By \eqref{PurelyOscProblem_HomoBrdData_GoodOprInterpolationEst1} and \eqref{PurelyOscProblem_HomoBrdData_GoodOprInterpolationEst2}, the
operator $\qoprgood$ extends uniquely to a bounded operator
\begin{align}\label{PurelyOscProblem_HomoBrdData_GoodOprInterpolationPoles1}
\begin{aligned}
&\qoprgood:\LR{q}\bp{\torus;\HSR{2}{q}\np\ws}^n\cap\HSR{1}{q}\bp{\torus;\LR{q}\np{\ws}}^n\ra\LR{\infty}\bp{\R_+;\LR{q}\np{\torus;\LR{q}(\ws)}}^n,\\
&\qoprgood:\LR{q}\bp{\torus;\HSR{1}{q}\np\ws}^n\cap\HSR{\frac{1}{2}}{q}\bp{\torus;\LR{q}\np{\ws}}^n\ra\LR{1,\infty}\bp{\R_+;\LR{q}\np{\torus;\LR{q}(\ws)}}^n. 
\end{aligned}
\end{align}
These extensions rely on the fact that $\SRh\np{\torus\times\ws}$ is dense in the function spaces on the left-hand side above.
We once more refer to
\cite[proof of Theorem 2.3.3 and Theorem 5.1.5]{TriebelTheoryFunctionSpaces} for a verification of this fact.
Using the projection $\projcompl$ on the left-hand side in \eqref{PurelyOscProblem_HomoBrdData_GoodOprInterpolationPoles1}, we obtain scales of the anisotropic Bessel-Potential spaces introduced in \eqref{BS_DefOfBesovSpace_DefnBessel}. Consequently, $\qoprgood$ is a bounded operator:
\begin{align}\label{PurelyOscProblem_HomoBrdData_GoodOprInterpolationPoles2}
\begin{aligned}
&\qoprgood:\ABPSRcompl{2}{q}\np\grpxx^n \ra\LR{\infty}\bp{\R_+;\LR{q}\np{\torus;\LR{q}(\ws)}}^n,\\
&\qoprgood:\ABPSRcompl{1}{q}\np\grpxx^n \ra\LR{1,\infty}\bp{\R_+;\LR{q}\np{\torus;\LR{q}(\ws)}}^n. 
\end{aligned}
\end{align}
Utilizing Lemma \ref{BS_InterpolationLem}, we find that
\begin{align*}
&\Bp{\ABPSRcompl{1}{q}\np\grpxx,\ABPSRcompl{2}{q}\np\grpxx}_{1-\frac{1}{q},q}=\BSRcompl{2-\frac{1}{q}}{qq}\np\grpxx\\
&\quad=\Bp{\projcompl\LR{q}\np{\torus;\LR{q}(\ws)},\ABPSRcompl{2}{q}\np\grpxx}_{1-\frac{1}{2q},q}\\
&\quad=\Bp{\projcompl\LR{q}\np{\torus;\LR{q}(\ws)}, \projcompl\LR{q}\bp{\torus;\HSR{2}{q}\np\ws}\cap\projcompl\HSR{1}{q}\bp{\torus;\LR{q}\np{\ws}}}_{1-\frac{1}{2q},q}\\
&\quad=\projcompl\Bp{\LR{q}\np{\torus;\LR{q}(\ws)}, \LR{q}\bp{\torus;\HSR{2}{q}\np\ws}}_{1-\frac{1}{2q},q}\\
&\qquad\qquad \cap \projcompl\Bp{\LR{q}\np{\torus;\LR{q}(\ws)},\HSR{1}{q}\bp{\torus;\LR{q}\np{\ws}}}_{1-\frac{1}{2q},q}\\
&\quad=\projcompl\WSR{1-\frac{1}{2q}}{q}\bp{\torus;\LR{q}(\ws)}\cap\projcompl\LR{q}\bp{\torus;\WSR{2-\frac{1}{q}}{q}(\ws)}.
\end{align*}
One can employ \cite[Theorem 1.12.1]{TriebelInterpolation} to verify the interpolation of the intersection space in the fourth equality above. 
Moreover, real interpolation yields
\begin{align*}
\Bp{ \LR{1, \infty}\bp{\R_+; \LR{q}\np{\torus;\LR{q}(\ws)}}, &\, \LR{\infty}\bp{\R_+; \LR{q}\np{\torus;\LR{q}(\ws)}}}_{1-\frac{1}{q}, q} \\
&\qquad\qquad\qquad\qquad\qquad = \LR{q}\bp{\R_+; \LR{q}\np{\torus;\LR{q}(\ws)}}.
\end{align*}
Recalling \eqref{PurelyOscProblem_HomoBrdData_GoodOprInterpolationPoles2}, we conclude that $\qoprgood$ extends uniquely to a bounded operator 
\begin{align}\label{PurelyOscProblem_HomoBrdData_GoodOprFinalMappingProperty}
\qoprgood: \projcompl\WSR{1-\frac{1}{2q}}{q}\bp{\torus;\LR{q}(\ws)}^n\cap\projcompl\LR{q}\bp{\torus;\WSR{2-\frac{1}{q}}{q}(\ws)}^n \ra
\LR{q}\bp{\torus;\LR{q}(\hs)}^n.
\end{align}
We now recall \eqref{PurelyOscProblem_HomoBrdData_DefOfBadOpr} and examine the operator $\qoprbad$. Utilizing \eqref{PurelyOscProblem_HomoBrdData_MultiplierOprNormIndep} with $m=0$, we obtain
\begin{align}\label{PurelyOscProblem_HomoBrdData_BadOprInterpolationPoles1}
\begin{aligned}
\norm{\qoprbad(\rhshn)}_{\LR{\infty}\np{\R_+;\LR{q}\np{\torus;\LR{q}(\ws)}}}
&\leq \Ccn{C}\, \norm{\iFT_{\torus\times\ws}\bb{\xi\qdftwo\np{k,\xi}}}_{\LR{q}\np{\torus;\LR{q}(\ws)}}\\
&\leq \Ccn{C}\, \norm{\rhshn}_{\HSR{1}{q}\np{\torus;\LR{q}\np{\ws}}}.
\end{aligned}
\end{align}
We again employ \eqref{PurelyOscProblem_HomoBrdData_MultiplierOprNormIndep} to estimate
\begin{align*}
\sup_{x_n>0}\,\norm{x_n \qoprbad\np{\rhshn}}_{\LR{q}\np{\torus; \LR{q}(\ws)}}\leq \Ccn{C}\,\norm{\rhshn}_{\HSR{1}{q}\np{\torus;\HSRD{-1}{q}\np{\ws}}},
\end{align*}
which implies
\begin{align*}
\norm{\qoprbad\np{\rhshn}}_{\LR{1,\infty}\np{\R_+;\LR{q}\np{\torus; \LR{q}(\ws)}}}
&= \normL{\frac{1}{x_n} \norm{x_n \qoprbad\np{\rhshn}}_{\LR{q}\np{\torus; \LR{q}(\ws)}}}_{\LR{1,\infty}\np{\R_+}}\\
&\leq \Ccn{C}\,\norm{\rhshn}_{\HSR{1}{q}\np{\torus;\HSRD{-1}{q}\np{\ws}}}.
\end{align*}
It follows that $\qoprbad$ extends to a bounded operator
\begin{align*}
&\qoprbad: \projcompl\HSR{1}{q}\bp{\torus;\LR{q}\np{\ws}} \ra\LR{\infty}\bp{\R_+;\LR{q}\np{\torus; \LR{q}(\ws)}}^n,\\
&\qoprbad: \projcompl\HSR{1}{q}\bp{\torus;\HSRD{-1}{q}\np{\ws}}\ra\LR{1,\infty}\bp{\R_+;\LR{q}\np{\torus; \LR{q}(\ws)}}.^n 
\end{align*}
Real interpolation thus implies that 
$\qoprbad$ extends to a bounded operator
\begin{align}\label{PurelyOscProblem_HomoBrdData_BadOprFinalMappingProperty}
\qoprbad: \projcompl\WSR{1}{q}\bp{\torus; \WSRD{-\frac{1}{q}}{q}(\ws)}\ra \LR{q}\bp{\torus; \LR{q}(\hs)}^n.
\end{align}
We now return to the solution formulas \eqref{PurelyOscProblem_HomoBrdData_SolFormula} and consider $\rhsh\in\projcompl\SRh\np{\torus\times\ws}^n$.
In this case, an application of \eqref{PurelyOscProblem_HomoBrdData_MultiplierOprNormIndep} ensures that $\upres$ is well-defined as an element
in the function space $\LR{q}\bp{\torus;\HSRD{1}{q}\np\hs}$.
By \eqref{PurelyOscProblem_HomoBrdData_GoodOprFinalMappingProperty} and \eqref{PurelyOscProblem_HomoBrdData_BadOprFinalMappingProperty},
we obtain
\begin{align}\label{PurelyOscProblem_HomoBrdData_EstimateForGradPresure}
\begin{aligned}
&\norm{\grad\upres}_{\LR{q}\np{\torus; \LR{q}(\hs)}} = \norm{\qoprgood\np{\rhsh}+\qoprbad\np{\rhshn}}_{\LR{q}\np{\torus; \LR{q}(\hs)}}\\
&\quad \leq \Ccn{C}
\bp{\norm{\rhsh}_{\WSR{1-\frac{1}{2q}}{q}\np{\torus;\LR{q}(\ws)}\cap\LR{q}\np{\torus;\WSR{2-\frac{1}{q}}{q}(\ws)}}
+\norm{\rhsh_n}_{\WSR{1}{q}(\torus; \WSRD{-\frac{1}{q}}{q}(\ws)}}
\end{aligned}
\end{align}
In a similar manner, it can be shown that $(\uvelp,\uveln)$ is well-defined as an element in the space $\LR{q}\bp{\torus;\HSR{2}{q}\np{\hs}}\cap\HSR{1}{q}\bp{\torus;\LR{q}\np{\hs}}$. To this end, one may consider the symbol
\begin{align*}
\mmultiplier:\R\times\ws\ra\C,\quad \mmultiplier\np{\eta,\xi}\ceq \bp{\sqrt{\snorm{\xi}^2+i\eta}\,x_n}^m \e^{-\sqrt{\snorm{\xi}^2 + i\eta} \, x_n} 
\end{align*}
and verify that 
\begin{align*}
\sup_{x_n>0}\sup_{\varepsilon\in\{0, 1\}^n}\sup_{(\eta,\xi)\in\R\times\ws}\left|\eta^{\varepsilon_0}\xi_1^{\varepsilon_1}\cdots\xi_{n-1}^{\varepsilon_{n-1}}\partial_{\eta}^{\varepsilon_0}\partial_{\xi_1}^{\varepsilon_1}\cdots\partial_{\xi_{n-1}}^{\varepsilon_{n-1}}
\mmultiplier(\eta,\xi) \right|<\infty.
\end{align*}
It follows that the Fourier-multiplier operator corresponding to the symbol $\mmultiplier$ is a bounded operator on $\LR{q}\np{\R; \LR{q}(\ws)}$ with operator norm independent on $x_n$. An application of the Transference Principle (Theorem \ref{transference}) therefore implies that the operator corresponding to the symbol $\mmultiplier_{|\Zgrp\times\ws}$ is a bounded operator on 
$\LR{q}\np{\torus; \LR{q}(\ws)}$ with operator norm independent on $x_n$, that is, 
\begin{align}\label{PurelyOscProblem_HomoBrdData_MultiplierOprNormIndep2}
\sup_{x_n>0}\,\normL{\phi\mapsto\FT_{\torus\times\ws}\bb{\mmultiplier(k,\xi)\FT_{\torus\times\ws}\nb{\phi}}}_{\linearmaps\np{\LR{q}\np{\torus; \LR{q}(\ws)},\LR{q}\np{\torus; \LR{q}(\ws)}}}<\infty.
\end{align}
With both \eqref{PurelyOscProblem_HomoBrdData_MultiplierOprNormIndep} and \eqref{PurelyOscProblem_HomoBrdData_MultiplierOprNormIndep2} at our disposal, it is now straightforward to verify that 
$\uvel\coloneqq(\uvelp,\uveln)$ is well-defined as element in the space $\LR{q}\bp{\torus;\HSR{2}{q}\np{\hs}}\cap\HSR{1}{q}\bp{\torus;\LR{q}\np{\hs}}$.
By construction, this choice of $(\uvel,\upres)$ is a solution to \eqref{PurelyOscProblem_HomoBrdData_Eq}.
Moreover, since $\projcompl\rhsh=\rhsh$, also $\projcompl\uvel=\uvel$. 
This means that $\uvel$ is a purely oscillatory solution in the aforementioned function space to the time-periodic heat equation in the half-space
\begin{align*}
\begin{pdeq}
\pt\uvel - \Delta\uvel &= -\grad\upres && \tin\torus\times\halfspace, \\
\uvel &= \rhsh && \ton\torus\times\partial\hs.
\end{pdeq}
\end{align*}
By \cite[Theorem 2.1]{KyedSauer_Heat} (see also \cite[Theorem 1.3]{KyedSauer_ADN1}), it is known that this problem has a unique purely oscillatory solution in the space
$\LR{q}\bp{\torus;\HSR{2}{q}\np{\hs}}\cap\HSR{1}{q}\bp{\torus;\LR{q}\np{\hs}}$, which satisfies
\begin{align*}
&\norm{\uvel}_{\HSR{1}{q}\np{\torus;\LR{q}(\hs)}\cap\LR{q}\np{\torus;\HSR{2}{q}(\hs)}}\\
&\qquad \leq \Ccn{C} \bp{\norm{\grad\upres}_{\LR{q}\np{\torus; \LR{q}(\hs)}} +
\norm{\rhsh}_{\WSR{1-\frac{1}{2q}}{q}\np{\torus;\LR{q}(\ws)}\cap\LR{q}\np{\torus;\WSR{2-\frac{1}{q}}{q}(\ws)}}}.
\end{align*}
Combining this estimate with \eqref{PurelyOscProblem_HomoBrdData_EstimateForGradPresure}, we conclude \eqref{PurelyOscProblem_HomoBrdData_Est}.
\end{proof}

In the next step, we consider the resolution of the fully non-homogeneous system \eqref{SHGP}, that is, \eqref{SH} with purely oscillatory data, and establish $\LR{q}$ estimates. This step concludes the main result of the article.

\begin{thm}\label{PurelyOscProblem_HomoDataThm}
Let $q\in(1,\infty)$ and $n\geq 2$. For all
\begin{align}
\begin{aligned}
&\f\in\projcompl\LR{q}\bp{\torus; \LR{q}(\hs)}^n,\\
&\g\in\projcompl\LR{q}\bp{\torus; \WSR{1}{q}(\halfspace)}\cap\projcompl\WSR{1}{q}\bp{\torus; \WSRD{-1}{q}(\halfspace)},\\
&\h\in\projcompl\WSR{1-\frac{1}{2q}}{q}\bp{\torus;\LR{q}(\ws)}^n\cap\projcompl\LR{q}\bp{\torus;\WSR{2-\frac{1}{q}}{q}(\ws)}^n 
\end{aligned}
\end{align}
with
\begin{align}
&\h_{n}\in\projcompl\WSR{1}{q}\bp{\torus; \WSRD{-\frac{1}{q}}{q}(\ws)}
\end{align}
there is a solution $(\uvel,\upres)$ to \eqref{SH} with 
\begin{align}
\begin{aligned}\label{PurelyOscProblem_HomoDataThm_SolReg}
&\uvel\in\projcompl\WSR{1}{q}\bp{\torus;\LR{q}(\hs)}^n\cap\projcompl\LR{q}\bp{\torus;\WSR{2}{q}(\hs)}^n,\\
&\upres\in\projcompl\LR{q}\bp{\torus;\WSRD{1}{q}(\hs)},
\end{aligned}
\end{align}
which satisfies
\begin{align}\label{PurelyOscProblem_HomoDataThm_ProjComplEst}
\begin{aligned}
&\norm{\uvel}_{\WSR{1}{q}\np{\torus;\LR{q}(\hs)}\cap\LR{q}\np{\torus;\WSR{2}{q}(\hs)}}
+ \norm{\grad\upres}_{\LR{q}\np{\torus; \LR{q}(\hs)}} \\
&\qquad \leq \Ccn{C}\, 
\bp{\norm{\f}_{\LR{q}\np{\torus; \LR{q}(\hs)}}+\norm{\g}_{\LR{q}\np{\torus; \WSR{1}{q}(\halfspace)}\cap\WSR{1}{q}\np{\torus; \WSRD{-1}{q}(\halfspace)}}\\
&\qquad\qquad +\norm{\h}_{\WSR{1-\frac{1}{2q}}{q}\np{\torus;\LR{q}(\ws)}\cap\LR{q}\np{\torus;\WSR{2-\frac{1}{q}}{q}(\ws)}}
+ \norm{\h_n}_{\WSR{1}{q}\np{\torus; \WSRD{-\frac{1}{q}}{q}(\ws)}}\,
}
\end{aligned}
\end{align}
with $\Ccn{C}=\Ccn{C}(n,q,\per)$. Moreover, if
$(\tuvel,\tupres)$ is another solution to \eqref{SH} in the class \eqref{PurelyOscProblem_HomoDataThm_SolReg}, then 
$\uvel=\tuvel$ and $\upres=\tupres + d(t)$
for some function $d$ that depends only on time.
\end{thm}

\begin{proof}
Let  
$\vvel\in\projcompl\WSR{1}{q}\bp{\torus;\LR{q}(\hs)}^n\cap\projcompl\LR{q}\bp{\torus;\WSR{2}{q}(\hs)}^n$ be
the solution to the purely oscillatory time-periodic n-dimensional heat equation
\begin{align*}
\begin{pdeq}
\pt\vvel - \Delta\vvel &= \f && \tin\torus\times\halfspace, \\
\vvel &= 0 && \ton\torus\times\partial\hs.
\end{pdeq}
\end{align*}
The existence of such a solution $\vvel$ that satisfies 
\begin{align*}
\norm{\vvel}_{\WSR{1}{q}\np{\torus;\LR{q}(\hs)}\cap\LR{q}\np{\torus;\WSR{2}{q}(\hs)}} \leq \Ccn{C} \norm{\f}_{\LR{q}\np{\torus; \LR{q}(\hs)}}
\end{align*}
follows from \cite[Theorem 2.1]{KyedSauer_Heat}. Denote by $\rhsG$ the extension of $\g-\Div\vvel$ to $\grp$ by even reflection in
the $x_n$ variable. Then
$\rhsG\in\projcompl\LR{q}\bp{\torus;\WSR{1}{q}\np{\Rn}}$. Moreover, identifying $\WSRD{-1}{q}\np{\Rn}$ as the dual of $\WSRD{1}{q'}\np{\Rn}$ and recalling that
$\g\in\projcompl\WSR{1}{q}\bp{\torus; \WSRD{-1}{q}(\halfspace)}$,
one directly verifies that $\rhsG\in\projcompl\WSR{1}{q}\bp{\torus;\WSRD{-1}{q}\np{\Rn}}$ with
\begin{multline*}
\norm{\rhsG}_{\WSR{1}{q}\np{\torus;\WSRD{-1}{q}\np{\Rn}}\cap\LR{q}\np{\torus;\WSR{1}{q}\np{\Rn}}}\\
\leq \Ccn{C}\bp{\norm{g}_{\WSR{1}{q}\np{\torus;\WSRD{-1}{q}\np{\hs}}\cap\LR{q}\np{\torus;\WSR{1}{q}\np{\hs}}}+
\norm{\vvel}_{\WSR{1}{q}\np{\torus;\LR{q}(\hs)}\cap\LR{q}\np{\torus;\WSR{2}{q}(\hs)}}}.
\end{multline*}
A solution
to the purely oscillatory Stokes system
\begin{align}\label{PurelyOscProblem_HomoDataThm_StokesRnReduction}
\begin{pdeq}
\pt\wvel-\Delta\wvel + \grad\wpres &=0 && \tin\grp,\\
\Div\wvel &= G &&\tin\grp
\end{pdeq}
\end{align}
is obtained via the solution formulas
\begin{align*}
\wvel \coloneqq \iFT_{\grp}\Bb{\frac{-i\xi}{\snorm{\xi}^2}\,\FT_{\grp}\nb{\rhsG}},\quad
\wpres \coloneqq \iFT_{\grp}\Bb{\frac{ik+\snorm{\xi}^2}{\snorm{\xi}^2}\,\FT_{\grp}\nb{\rhsG}}.
\end{align*}
From these formulas, we immediately obtain the estimate
\begin{align*}
\norm{\wvel}_{\WSR{1}{q}\np{\torus;\LR{q}(\Rn)}\cap\LR{q}\np{\torus;\WSR{2}{q}(\Rn)}}
&+ \norm{\grad\wpres}_{\LR{q}\np{\torus; \LR{q}(\R^n)}} \\
&\qquad\qquad \leq \Ccn{C} \norm{\rhsG}_{\LR{q}\np{\torus;\WSR{1}{q}\np{\Rn}}\cap\WSR{1}{q}\np{\torus;\WSRD{-1}{q}\np{\Rn}}}.
\end{align*}
By the symmetry of $\rhsG$, the vector field $\twvel$ obtained by odd reflection with respect to $x_n$ of the n'th component of $\wvel$, that is,
\begin{align*}
\twvel\np{t,x',x_n}\coloneqq \bp{\wvel_1\np{t,x',x_n},\ldots,\wvel_{n-1}\np{t,x',x_n},-\wvel_n\np{t,x',-x_n}}, 
\end{align*}
is also a solution to \eqref{PurelyOscProblem_HomoDataThm_StokesRnReduction} corresponding to the same pressure term $\wpres$.
This means that $\wvel$ and $\twvel$ both solve the same time-periodic heat equation in the whole-space $\grp$. By
\cite[Theorem 2.1]{KyedSauer_Heat}, $\wvel=\twvel$. It follows that
$\trace_{\torus\times\ws}\nb{\wvel_n}=0$. Consequently, $\rhsh\coloneqq\h-\trace_{\torus\times\ws}\nb{\wvel}$ belongs
to the space \eqref{PurelyOscProblem_HomoBrdData_DataReg} (see for example \cite{KyedSauer_ADN1} for a rigorous definition of the trace operator in this setting).
Let $(\Uvel,\Upres)$  be the corresponding solution from Proposition \ref{PurelyOscProblem_HomoBrdData}. It follows that
$\np{\uvel,\upres}\coloneqq\np{\Uvel+\wvel+\vvel,\Upres+\wpres}$ is a solution
to \eqref{SH} in the class \eqref{PurelyOscProblem_HomoDataThm_SolReg} satisfying \eqref{PurelyOscProblem_HomoDataThm_ProjComplEst}.

It remains to show uniqueness, which follows from a standard duality argument. To this end, assume that $(\tuvel,\tupres)$ is a solution
in the class \eqref{PurelyOscProblem_HomoDataThm_SolReg} to the homogeneous Stokes problem
\begin{align*}
\begin{pdeq}
\pt\tuvel - \Delta\tuvel + \grad \tupres &= 0 && \tin\torus\times\halfspace, \\
\Div\tuvel &= 0 && \tin\torus\times\halfspace, \\
\tuvel &= 0 && \ton\torus\times\partial\halfspace.
\end{pdeq}
\end{align*}
Let $\phi\in\CRci\np{\torus\times\hs}^n$. With exactly the same arguments as above, one can establish existence of a solution
\begin{align*}
&\adjvel\in\projcompl\WSR{1}{q\prime}\bp{\torus;\LR{q\prime}(\hs)}^n\cap\projcompl\LR{q\prime}\bp{\torus;\WSR{2}{q\prime}(\hs)}^n,\\
&\adjpres\in\projcompl\LR{q\prime}\bp{\torus;\WSRD{1}{q\prime}(\hs)},
\end{align*}
to the adjoint Stokes problem
\begin{align*}
\begin{pdeq}
\pt\adjvel + \Delta\adjvel + \grad \adjpres &= \phi && \tin\torus\times\halfspace, \\
\Div\adjvel &= 0 && \tin\torus\times\halfspace, \\
\adjvel &= 0 && \ton\torus\times\partial\halfspace,
\end{pdeq}
\end{align*}
where $q\prime$ denotes the H\"older conjugate of $q$.
Integration by parts yields
\begin{align*}
\int_\torus\int_\hs \tuvel\cdot\phi\,\dx\dt =
\int_\torus\int_\hs \tuvel\cdot\bp{\pt\adjvel + \Delta\adjvel + \grad \adjpres}\,\dx\dt= 0.
\end{align*}
Since this identity holds for all $\phi\in\CRci\np{\torus\times\hs}^n$, it follows that $\tuvel=0$. In turn, we deduce $\grad\tupres=0$, whence
$\tupres\in\projcompl\LR{q}\np{\torus}$, that is, $\tupres$ depends only on time. 
\end{proof}

\begin{proof}[Proof of Theorem \ref{MainThm}]
Let $\f,\g,\h$ be vector fields in the class \eqref{MainThm_Data}, with $\h$ satisfying \eqref{MainThm_DataCompCond}. 
By \cite[Theorem IV.3.2]{Galdi}, the steady-state Stokes problem \eqref{SHGS}
admits a solution $\np{\vvel,\vpres}\in\WSRD{2}{q}\np{\hs}^n\times\WSRD{1}{q}\np{\hs}$ that satisfies 
\begin{multline*}
\norm{\grad^2\vvel}_{\LR{q}(\hs)} + \norm{\grad\vpres}_{\LR{q}(\hs)}\leq \Ccn{C}
\bp{\norm{\proj \f}_{\LR{q}(\hs)} + \norm{\proj\g}_{\WSR{1}{q}(\hs)}+\norm{\proj\h}_{\WSR{2-\frac{1}{q}}{q}(\ws)}}.
\end{multline*}
By Theorem \ref{PurelyOscProblem_HomoDataThm}, the purely oscillatory Stokes problem \eqref{SHGP} admits a solution 
$(\wvel,\wpres)$ in the class \eqref{PurelyOscProblem_HomoDataThm_SolReg} satisfying \eqref{PurelyOscProblem_HomoDataThm_ProjComplEst}.
Putting $(\uvel,\upres)\coloneqq(\vvel+\wvel,\vpres+\wpres)$, we obtain a solution to \eqref{SH} that satisfies \eqref{MainThm_ProjEst} and
\eqref{MainThm_ProjComplEst}. Finally, 
if $(\tuvel,\tupres)$ is another solution to \eqref{SH} in the class \eqref{MainThm_SolReg}, then 
$\projcompl\uvel=\projcompl\tuvel$ by Theorem \ref{PurelyOscProblem_HomoDataThm}, and $\proj\uvel=\proj\tuvel+(a_1x_n,\ldots,a_{n-1}x_{n},0)$ for some vector $a\in\R^{n-1}$ by \cite[Theorem IV.3.2]{Galdi}. It follows that $\grad\upres=\grad\tupres$, and thus $\upres=\tupres + d(t)$
for some function $d$ that depends only on time.
\end{proof}

\bibliographystyle{plainurl}

\end{document}